\documentclass[12pt]{amsart}%
\usepackage{amssymb}
\usepackage{amsmath}
\usepackage{amsfonts}
\usepackage{geometry}
\usepackage{graphicx}%
\providecommand{\U}[1]{\protect\rule{.1in}{.1in}}
\newtheorem{theorem}{Theorem}[section]
\theoremstyle{plain}

\newtheorem{corollary}{Corollary}[section]

\newtheorem{lemma}{Lemma}[section]

\newtheorem{proposition}{Proposition}[section]
\newtheorem{remark}{Remark}[section]

\numberwithin{equation}{section}
\begin{document}
\title[Elliptic equations with critical growth in $\mathbb{R}^{N}$]{Existence and multiplicity of solutions to equations of $N-$Laplacian type
with critical exponential growth in $\mathbb{R}^{N}$}
\author{Nguyen Lam}
\author{Guozhen Lu}
\address{Nguyen Lam and Guozhen Lu\\
Department of Mathematics\\
Wayne State University\\
Detroit, MI 48202, USA\\
Emails: nguyenlam@wayne.edu and gzlu@math.wayne.edu}
\thanks{Corresponding Author: G. Lu at gzlu@math.wayne.edu}
\thanks{Research is partly supported by a US NSF grant DMS0901761.}
\date{}

\begin{abstract}
In this paper, we deal with the existence of solutions to the nonuniformly
elliptic equation of the form
\begin{equation}
-\operatorname{div}\left(  a\left(  x,\nabla u\right)  \right)
+V(x)\left\vert u\right\vert ^{N-2}u=\frac{f(x,u)}{\left\vert x\right\vert
^{\beta}}+\varepsilon h(x) \label{0.1}%
\end{equation}
in $\mathbb{R}^{N}$ where $0\leq\beta<N$, $V:\mathbb{R}^{N}\rightarrow
\mathbb{R}$ is a continuous potential satisfying $V(x)\ge V_0>0$ in $\mathbb{R}^N$ and $V^{-1}\in L^1(\mathbb{R}^N)$ or $|\{x\in \mathbb{R}^N: V(x)\le M\}|<\infty$ for every $M>0$, $f:\mathbb{R}^{N}\times\mathbb{R}%
\rightarrow\mathbb{R}$ behaves like $\exp\left(  \alpha\left\vert
u\right\vert ^{N/(N-1)}\right)  $ when $\left\vert u\right\vert
\rightarrow\infty$ and satisfies the Ambrosetti-Rabinowitz
condition, $h\in\left(  W^{1,N}\left( \mathbb{R}^{N}\right)  \right)
^{\ast},~h\neq0$ and $\varepsilon$ is a positive parameter.
 In
particular, in the case of $N-$Laplacian, i.e,
\begin{equation}
-\Delta_{N}u+V(x)\left\vert u\right\vert ^{N-2}u=\frac{f(x,u)}{\left\vert
x\right\vert ^{\beta}}+\varepsilon h(x) \label{0.2}%
\end{equation}
using the minimization and the Ekeland variational principle, we obtain
multiplicity of weak solutions of (\ref{0.2}).

Moreover, we prove that it is not necessary to have the small nonzero
perturbation $\varepsilon h(x)$ to get the nontriviality of the solution to
the $N-$Laplacian equation
\begin{equation}
-\Delta_{N}u+V(x)\left\vert u\right\vert ^{N-2}u=\frac{f(x,u)}{\left\vert
x\right\vert ^{\beta}}%
\end{equation}

Finally, we will prove the above results when our nonlinearity $f$
doesn't satisfy the well-known Ambrosetti-Rabinowitz condition and
thus derive the existence and multiplicity of solutions for a  wider
class of nonlinear terms $f$.

\end{abstract}
\maketitle

\section{Introduction}

In this paper,  we consider the existence and multiplicity of nontrivial weak
solution $u\in W^{1, N}(\mathbb{R}^N)$ ($u\ge 0$) for the nonuniformly elliptic equations of $N-$Laplacian type of the
form:%
\begin{equation}
-\operatorname{div}\left(  a\left(  x,\nabla u\right)  \right)
+V(x)\left\vert u\right\vert ^{N-2}u=\frac{f(x,u)}{\left\vert x\right\vert
^{\beta}}+\varepsilon h(x) \,\,\, in \,\, \mathbb{R}^N \label{1.1}%
\end{equation}
where, in addition to some more assumptions on $a(x, \tau)$ and $f$ which will be specified later in Section 2, we have
\[
\left\vert a\left(  x,\tau\right)  \right\vert \leq c_{0}\left(  h_{0}\left(
x\right)  +h_{1}\left(  x\right)  \left\vert \tau\right\vert ^{N-1}\right)
\]
for any $\tau$ in $\mathbb{R}^{N}$ and a.e. $x$ in $\mathbb{R}^{N}$, $h_{0}\in L^{N/(N-1)}\left(
\mathbb{R}^{N}\right)$ and $h_{1}\in L_{loc}^{\infty}\left(  \mathbb{R}%
^{N}\right)  $ and
$f$ satisfies critical growth of exponential type such as $f:\mathbb{R}^{N}\times\mathbb{R}%
\rightarrow\mathbb{R}$ behaves like $\exp\left(  \alpha\left\vert u\right\vert
^{N/(N-1)}\right)  $ when $\left\vert u\right\vert \rightarrow\infty$ and when $f$ either satisfies or does not satisfy the Ambrosetti-Rabinowitz condition.

 A special case of our
equation in the whole Euclidean space when $a\left(  x,\nabla u\right)  =\left\vert \nabla u\right\vert
^{N-2}\nabla u$ has been studied extensively, both in the case $N=2$ (the prototype equation is the Laplacian in $\mathbb{R}^2$) and in the case $N>3$ in
 $\mathbb{R}^N$ for the $N-$Laplacain,  see  for example \cite{C}, \cite{A},  \cite{AS}, \cite{P}, \cite{FMR}, \cite{Do1, Do, DoMS}, \cite{AY}, etc.
We should mention that problems involving Laplacian in bounded domains in $\mathbb{R}^2$ with critical exponential growth have been studied in \cite{AY1}, \cite{FMR}, \cite{AP1}, \cite{AP2}, \cite{CC}, \cite{S}, etc. and for $N-$Laplacian in bounded domains in  $\mathbb{R}^N$  ($N>2$)   by the authors of \cite{A}, \cite{Do1}, \cite{P}.

The problems of this type are important in many fields of sciences, notably the
fields of electromagnetism, astronomy, and fluid dynamics, because they can be
used to accurately describe the behavior of electric, gravitational, and fluid
potentials. They have been extensively studied by many authors in many
different cases: bounded domains and unbounded domains, different behaves of
the nonlinearity, different types of boundary conditions, etc. In particular,
many works focus on the subcritical and critical growth of the nonlinearity
which allows to treat the problem variationally using general critical point
theory.

In the case $p<N$, by the Sobolev embedding, the subcritical and
critical growth for the $p-$Laplacian mean that the nonlinearity $f$
cannot exceed the polynomial of degree $p^{\ast}=\frac{Np}{N-p}$.
The case $p=N$ is special, since the corresponding Sobolev space $W_{0}%
^{1,N}\left(  \Omega\right)  $ is a borderline case for Sobolev embeddings:
one has $W_{0}^{1,N}\left(  \Omega\right)  \subset L^{q}\left(  \Omega\right)
$ for all $q\geq1$, but $W_{0}^{1,N}\left(  \Omega\right)  \nsubseteq
L^{\infty}\left(  \Omega\right)  $. So, one is led to ask if there is another
kind of \textit{maximal growth} in this situation. Indeed, this is the result
of Pohozaev \cite{Po}, Trudinger \cite{T} and Moser \cite{M}, and is by now
called the Moser-Trudinger inequality: it says that if $\Omega\subset%
\mathbb{R}
^{N}$ is a bounded domain, then
\[
\underset{u\in W_{0}^{1,N}\left(  \Omega\right)  ,~\left\Vert \nabla
u\right\Vert _{L^{N}}\leq1}{\sup}\frac{1}{|\Omega|}\int_{\Omega}e^{\alpha_{N}
|u|^{\frac{N}{N-1}}}dx <\infty
\]
where $\alpha_{N}=Nw_{N-1}^{\frac{1}{N-1}}$ and $w_{N-1}$ is the surface area
of the unit sphere in $\mathbb{R}^{N}$. Moreover, the constant $\alpha_{N}$ is
sharp in the sense that if we replace $\alpha_{N}$ by some $\beta>\alpha_{N}$,
the above supremum is infinite.

This well-known Moser-Trudinger inequality has been generalized in many ways.
For instance, in the case of bounded domains, Adimurthi and Sandeep proved in
\cite{AS} that the following inequality
\[
\underset{u\in W_{0}^{1,N}\left(  \Omega\right)  ,~\left\Vert \nabla
u\right\Vert _{L^{N}}\leq1}{\sup}\int_{\Omega}\frac{e^{\alpha_{N}
|u|^{\frac{N}{N-1}}}}{|x|^{\beta}}dx <\infty
\]
holds if and only if $\frac{\alpha}{\alpha_{N}}+\frac{\beta}{N}\le1$ where
$\alpha>0$ and $0\le\beta<N$.

On the other hand, in the case of unbounded domains,
B. Ruf when $N=2$ in
\cite{R} and Y. X. Li and B. Ruf when $N>2$ in \cite{LR} proved that if we
replace the $L^{N}$-norm of $\nabla u$ in the supremum by the standard Sobolev
norm, then this supremum can still be finite under a certain condition for
$\alpha$. More precisely, they have proved the following:
\[
\underset{u\in W_{0}^{1,N}\left(
\mathbb{R}
^{N}\right)  ,~\left\Vert u\right\Vert _{L^{N}}^{N}+\left\Vert \nabla
u\right\Vert _{L^{N}}^{N}\leq1}{\sup}\int_{%
\mathbb{R}
^{N}}\left(  \exp\left(  \alpha\left\vert u\right\vert ^{N/(N-1)}\right)
-S_{N-2}\left(  \alpha,u\right)  \right)  dx\left\{
\begin{array}
[c]{c}%
\leq \infty\text{ if }\alpha\leq\alpha_{N},\\
=+\infty\text{ if }\alpha>\alpha_{N}%
\end{array}
\right.  ,
\]
where
\[
S_{N-2}\left(  \alpha,u\right)  =\sum\limits_{k=0}^{N-2}\frac{\alpha
^{k}\left\vert u\right\vert ^{kN/(N-1)}}{k!}.
\]

We should mention that for $\alpha<\alpha_{N}$ when $N=2$, the above
inequality was first proved by D. Cao in \cite{C}, and proved  for
$N>2$ by Panda \cite{P} and J.M. do O \cite{Do1, Do} and Adachi and Tanaka \cite{AT}.

Recently, Adimurthi and Yang generalized the above result of Li and Ruf
\cite{LR} to get the following version of the singular Trudinger-Moser
inequality (see \cite{AY}):

\begin{lemma}
For all $0\leq\beta<N,~0<\alpha$ and $u\in W^{1,N}\left(  \mathbb{R}%
^{N}\right)  $, there holds%
\[
\int_{\mathbb{R}^{N}}\frac{1}{\left\vert x\right\vert ^{\beta}}\left\{
\exp\left(  \alpha\left\vert u\right\vert ^{N/(N-1)}\right)  -S_{N-2}\left(
\alpha,u\right)  \right\}  <\infty
\]
Furthermore, we have for all $\alpha\leq\left(  1-\frac{\beta}{N}\right)
\alpha_{N}$ and $\tau>0$,
\[
\underset{\left\Vert u\right\Vert _{1,\tau}\leq1}{\sup}\int_{\mathbb{R}^{N}%
}\frac{1}{\left\vert x\right\vert ^{\beta}}\left\{  \exp\left(  \alpha
\left\vert u\right\vert ^{N/(N-1)}\right)  -S_{N-2}\left(  \alpha,u\right)
\right\}  <\infty
\]
where $\left\Vert u\right\Vert _{1,\tau}=\left(  \int_{\mathbb{R}^{N}}\left(
\left\vert \nabla u\right\vert ^{N}+\tau\left\vert u\right\vert ^{N}\right)
dx\right)  ^{1/N}$. The inequality is sharp: for any $\alpha>\left(
1-\frac{\beta}{N}\right)  \alpha_{N}$, the supremum is infinity.
\end{lemma}

Motivated by this Trudinger-Moser inequality, do \'{O} \cite{Do1, Do} and do \'{O},  Medeiros and  Severo \cite{DoMS} studied the
quasilinear elliptic equations  when $\beta=0$ and
Adimurthi and Yang \cite{AY}
studied the singular quasilinear elliptic equations for $0\le \beta<N$, both with the maximal growth on
the singular nonlinear term $\frac{f(x,u)}{\left\vert x\right\vert ^{\beta}}$
which allows them to treat the equations variationally in a subspace of
$W^{1,N}\left(  \mathbb{R}^{N}\right)  $. More precisely, they can find a
nontrivial weak solution of mountain-pass type to the equation with the
perturbation
\[
-\operatorname{div}\left(  \left\vert \nabla u\right\vert ^{N-2}\nabla
u\right)  +V(x)\left\vert u\right\vert ^{N-2}u=\frac{f(x,u)}{\left\vert
x\right\vert ^{\beta}}+\varepsilon h(x)
\]
Moreover, they proved that when the positive parameter $\varepsilon$ is small
enough, the above equation has a weak solution with negative energy. However, it was not proved in \cite{AY}
if those solutions are different or not. We also should
stress that they need a small nonzero perturbation $\varepsilon h(x)$ in their
equation to get the nontriviality of the solutions.

In this paper, we will study further about the equation considered in the whole space
\cite{A, Do1, Do, DoMS, AY}. More precisely, we consider the existence and multiplicity of nontrivial weak
solution for the nonuniformly elliptic equations of $N-$Laplacian type of the
form:%
\begin{equation}
-\operatorname{div}\left(  a\left(  x,\nabla u\right)  \right)
+V(x)\left\vert u\right\vert ^{N-2}u=\frac{f(x,u)}{\left\vert x\right\vert
^{\beta}}+\varepsilon h(x) \label{1.1}%
\end{equation}
where%
\[
\left\vert a\left(  x,\tau\right)  \right\vert \leq c_{0}\left(  h_{0}\left(
x\right)  +h_{1}\left(  x\right)  \left\vert \tau\right\vert ^{N-1}\right)
\]
for any $\tau$ in $%
\mathbb{R}
^{N}$ and a.e. $x$ in $\mathbb{R}^{N}$, $h_{0}\in L^{N/(N-1)}\left(
\mathbb{R}^{N}\right)  $ and $h_{1}\in L_{loc}^{\infty}\left(  \mathbb{R}%
^{N}\right)  $. Note that the equation in \cite{AY} is a special case of our
equation when $a\left(  x,\nabla u\right)  =\left\vert \nabla u\right\vert
^{N-2}\nabla u$. In fact, the elliptic equations of nonuniform type is a
natural generalization of the $p-$Laplacian equation and were studied by many
authors, see \cite{DV, G, HN, HZ, Te}. As mentioned earlier, the main features
of this class of equations are that they are defined in the whole
$\mathbb{R}^{N}$ and with the critical growth of the singular nonlinear term
$\frac{f(x,u)}{\left\vert x\right\vert ^{\beta}}$ and the nonuniform nonlinear
operator of $p$-Laplacian type. In spite of a possible failure of the
Palais-Smale compactness condition, in this paper, we still use the
Mountain-pass approach for the critical growth as in \cite{Do1, AY, Do, DoMS} to
derive a weak solution and get the nontriviality of this solution thanks to
the small nonzero perturbation $\varepsilon h(x)$.

In the case of $N-$Laplacian, i.e.,
\[
a\left(  x,\nabla u\right)  =\left\vert \nabla u\right\vert ^{N-2}\nabla u,
\]
our equation is exactly the equation studied in \cite{AY}:
\begin{equation}
-\operatorname{div}\left(  \left\vert \nabla u\right\vert ^{N-2}\nabla
u\right)  +V(x)\left\vert u\right\vert ^{N-2}u=\frac{f(x,u)}{\left\vert
x\right\vert ^{\beta}}+\varepsilon h(x) \label{1.2}%
\end{equation}
Using the Radial lemma, Schwarz  symmetrization and a modified
result of Lions \cite{L} about the singular Moser-Trudinger
inequality, we will prove that two solutions derived in \cite{AY}
are actually different. Thus as our second main result, we get the
multiplicity of solutions to the equation (\ref{1.2}). This result
extends the result in \cite{AY} and also the multiplicity result in
\cite{DoMS} ($\beta=0$) to the singular case ($0\le \beta <N$).

Our next concern is about the existence of solution of the equation without
the perturbation%
\begin{equation}
-\operatorname{div}\left(  \left\vert \nabla u\right\vert ^{N-2}\nabla
u\right)  +V(x)\left\vert u\right\vert ^{N-2}u=\frac{f(x,u)}{\left\vert
x\right\vert ^{\beta}}. \label{1.3}%
\end{equation}
Using an approach as in \cite{Do1, Do, DoMS}, we prove that we don't even
require the nonzero perturbation as in \cite{AY} to get the nontriviality of
the mountain-pass type weak solution.

Our main tool in this paper is critical point theory. More
precisely, we will use the Mountain-pass Theorem that is proposed by
Ambrosetti and Rabinowitz in the celebrated paper \cite{AR}.
Critical point theory has become one of the main tools for finding
solutions to elliptic equations of variational type. We stress that
to use the Mountain-pass Theorem, we need to verify some types of
compactness for the associated Lagrange-Euler functional, namely the
Palais-Smale condition and the Cerami condition. Or at least, we
must prove the boundedness of the Palais-Smale or Cerami sequence
\cite{Ce1, Ce2}. In almost all of works, we can easily establish
this condition thanks to the Ambrosetti-Rabinowitz (AR) condition,
see $(f2)$ or $(f3)$ in Section 2. However, there are many
interesting examples of nonlinear terms $f$ which do not satisfy the
Ambrosetti-Rabinowitz condition, but based on our theorem we can
still conclude the existence and multiplicity of solutions. Thus our
next result is that we will establish again the above results when
the nonlinearity does not satisfy this famous (AR) condition. For
the $N-$Laplacian equation in a bounded domain in $\mathbb{R}^{N}$,
such a result of existence has been established by the authors in
\cite{LaLu}.

We mention in passing that the study of the existence and
multiplicity results of nonuniformly elliptic equations of
$N-$Laplacian type are motivated by our earlier work on the
Heisenberg group \cite{LaLu2}. Our assumptions on the potential $V$
are exactly those considered in  \cite{Do1, Do, DoMS, AY}, namely
$V(x)\ge V_0>0$ in $\mathbb{R}^N$ and $V^{-1}\in L^1(\mathbb{R}^N)$
or $|\{x\in \mathbb{R}^N: V(x)<M\}|<\infty$ for every $M>0$. Very
recently, Yang has established in \cite{Y} when $a(x, \nabla
u)=|\nabla u|^{N-2}\nabla u$ the multiplicity of solutions when the
nonlinear term $f$ satisfies the Ambrosetti-Rabinowitz condition and
the potential $V$ is under a stronger assumption than ours. More
precisely, it is assumed in \cite{Y} that $V^{-1}\in
L^{\frac{1}{N-1}}(\mathbb{R}^{N})$ which implies $V^{-1}\in
L^{1}(\mathbb{R}^{N})$ when $V(x)\ge V_{0}>0$ in $\mathbb{R}^{N}$.
The stronger assumption of integrability on $V^{-1}$ in \cite{Y}
guarantees that the embedding $E\rightarrow L^{q}(\mathbb{R}^{N})$
is compact for all $1\leq q<\infty$. The argument in \cite{Y}, as
pointed out by the author of \cite{Y}, depends crucially on this
compact embedding for all $1\leq q<\infty$. The assumption on the
potential $V$ in our paper only assures the compact embedding
$E\rightarrow L^{q}(\mathbb{R}^{N})$ for $q\geq N$. Nevertheless,
this compact embedding for $q\geq N$ is sufficient for us to carry
out the proof of the multiplicity of solutions to equation
(\ref{1.2}) and existence of solutions to equation (\ref{1.3})
without the perturbation term. (See Proposition \ref{prop5.2} and
Remark \ref{rem5.2} in Section 5 for more details). Moreover, our
theorems hold even when $f$ does not satisfy the
Ambrosetti-Rabinowitz condition.

The paper is organized as follows: In the next section, we give the main
assumptions which are used throughout this paper except the last section and
our main results. In Section 3, we prove some preliminary results. Section 4
is devoted to study the existence of nontrivial solutions for the nonuniformly
elliptic equations of $N-$Laplacian type (\ref{1.1}). The multiplicity of
nontrivial solutions to Equation (\ref{1.2}) is investigated in Section 5.
Section 6 is about the existence of nontrivial solutions to the equation
without the perturbation (\ref{1.3}). Finally, in Section 7 we study the
results in Sections 5 and 6 again without the well-known Ambrosetti-Rabinowitz
(AR) condition.

\section{Assumptions and Main Results}

Motivated by the Trudinger-Moser inequality in Lemma 1.1, we consider here the
maximal growth on the nonlinear term $f(x,u)$ which allows us to treat
Eq.(\ref{1.1}) variationally in a subspace of $W^{1,N}\left(  \mathbb{R}%
^{N}\right)  $. We assume that $f:\mathbb{R}^{N}\times\mathbb{R}%
\rightarrow\mathbb{R}$ is continuous, $f(x,0)=0$ and $f$ behaves like
$\exp\left(  \alpha\left\vert u\right\vert ^{N/(N-1)}\right)  $ as $\left\vert
u\right\vert \rightarrow\infty$. More precisely, we assume the following
growth conditions on the nonlinearity $f(x,u)$ as in \cite{Do1, Do, DoMS, AY}:

$(f1)$ There exist constants $\alpha_{0},~b_{1},~b_{2}>0$ such that for all
$\left(  x,u\right)  \in\mathbb{R}^{N}\times\mathbb{R}^{+}$,
\[
0<f(x,u)\leq b_{1}\left\vert u\right\vert ^{N-1}+b_{2}\left[  \exp\left(
\alpha_{0}\left\vert u\right\vert ^{N/(N-1)}\right)  -S_{N-2}\left(
\alpha_{0},u\right)  \right]  ,
\]
where
\[
S_{N-2}\left(  \alpha_{0},u\right)  =\sum\limits_{k=0}^{N-2}\frac{\alpha
_{0}^{k}}{k!}\left\vert u\right\vert ^{kN/(N-1)}.
\]

$(f2)$ There exist $p>N$ such that for all $x\in\mathbb{R}^{N}$ and $s>0$,
\[
0<pF(x,s)=p%
{\displaystyle\int\limits_{0}^{s}}
f(x,\tau)d\tau\leq sf(x,s)
\]
This is the well-known Ambrosetti-Rabinowitz condition.

$(f3)$ There exist constants $R_{0},~M_{0}>0$ such that for all $x\in
\mathbb{R}^{N}$ and $s\geq R_{0}$,
\[
F(x,s)\leq M_{0}f(x,s).
\]
Since we are interested in nonnegative weak solutions, it is convenient to
define
\begin{equation}
f(x,u)=0\text{ for all }\left(  x,u\right)  \in\mathbb{R}^{N}\times\left(
-\infty,0\right]  . \label{2.1}%
\end{equation}

Let $\ A$ be a measurable function on $\mathbb{R}^{N}\times\mathbb{R}$ such
that $A(x,0)=0$ and $a(x,\tau)=\frac{\partial A\left(  x,\tau\right)
}{\partial\tau}$ is a Caratheodory function on $\mathbb{R}^{N}\times
\mathbb{R}$. Assume that there are positive real numbers $c_{0},c_{1},k_{1}$
and two nonnegative measurable functions $h_{0},~h_{1}$ on $\mathbb{R}^{N}$
such that $h_{1}\in L_{loc}^{\infty}\left(  \mathbb{R}^{N}\right)  ,\ h_{0}\in
L^{N/(N-1)}\left(  \mathbb{R}^{N}\right)  ,~h_{1}(x)\geq1$ for a.e. $x$ in
$\mathbb{R}^{N}$ and the following conditions hold:
\[
\left.
\begin{array}
[c]{l}%
\left(  A1\right)  ~\left\vert a(x,\tau)\right\vert \leq c_{0}\left(
h_{0}\left(  x\right)  +h_{1}\left(  x\right)  \left\vert \tau\right\vert
^{N-1}\right)  ,\,~\forall\tau\in%
\mathbb{R}
^{N},~a.e.\,x\in\mathbb{R}^{N}\\
(A2)~c_{1}\left\vert \tau-\tau_{1}\right\vert ^{N}\leq\left\langle
a(x,\tau)-a(x,\tau_{1}),\tau-\tau_{1}\right\rangle \,~\forall\tau,\tau_{1}\in%
\mathbb{R}
^{N},~a.e.\,\,x\in\mathbb{R}^{N}\\
(A3)~0\leq a(x,\tau).\tau\leq NA\left(  x,\tau\right)  \text{ }\,\,\forall
\tau\in%
\mathbb{R}
^{N},~a.e.\,x\in\mathbb{R}^{N}\\
(A4)~A\left(  x,\tau\right)  \geq k_{0}h_{1}\left(  x\right)  \left\vert
\tau\right\vert ^{N}~\,\forall\tau\in%
\mathbb{R}
^{N},~a.e.\,x\in\mathbb{R}^{N}.
\end{array}
~\right.
\]
Then $A$ verifies the growth condition:
\begin{equation}
\left\vert A\left(  x,\tau\right)  \right\vert \leq c_{0}\left(  h_{0}\left(
x\right)  \left\vert \tau\right\vert +h_{1}\left(  x\right)  \left\vert
\tau\right\vert ^{N}\right)  \,~\forall\tau\in%
\mathbb{R}
^{N},~a.e.\,x\in\mathbb{R}^{N} \label{2.2}%
\end{equation}
Next, we introduce some notations:%

\[%
\begin{array}
[c]{l}%
E=\left\{  u\in\bigskip W_{0}^{1,N}(\mathbb{R}^{N}):\int_{\mathbb{R}^{N}}%
h_{1}(x)\left\vert \nabla u\right\vert ^{N}dx+\int_{\mathbb{R}^{N}%
}V(x)\left\vert u\right\vert ^{N}<\infty\right\} \\
\left\Vert u\right\Vert _{E}=\left(  \int_{\mathbb{R}^{N}}\left(
h_{1}(x)\left\vert \nabla u\right\vert ^{N}+\frac{1}{k_{0}N}V(x)\left\vert
u\right\vert ^{N}\right)  dx\right)  ^{1/N},~u\in\bigskip E\\
\lambda_{1}\left(  N\right)  =\inf\left\{  \frac{\left\Vert u\right\Vert
_{E}^{N}}{\int_{\mathbb{R}^{N}}\frac{\left\vert u\right\vert ^{N}}{\left\vert
x\right\vert ^{\beta}}dx}:u\in\bigskip E\setminus\left\{  0\right\}  \right\}
\end{array}
\]

\bigskip We also assume the following conditions on the potential as in \cite{Do1, Do, DoMS, AY}:

$(V1)$ $V$ is a continuous function such that $V(x)\geq V_{0}>0$ for all
$x\in\mathbb{R}^{N}$, we can see that $E$ is a reflexive Banach space when
endowed with the norm
\[
\left\Vert u\right\Vert _{E}=\left(  \int_{\mathbb{R}^{N}}\left(
h_{1}(x)\left\vert \nabla u\right\vert ^{N}+\frac{1}{k_{0}N}V(x)\left\vert
u\right\vert ^{N}\right)  dx\right)  ^{1/N}%
\]
and for all $N\leq q<\infty$,
\[
E\hookrightarrow W^{1,N}\left(  \mathbb{R}^{N}\right)  \hookrightarrow
L^{q}\left(  \mathbb{R}^{N}\right)
\]
with continuous embedding. Furthermore,
\begin{equation}
\lambda_{1}\left(  N\right)  =\inf\left\{  \frac{\left\Vert u\right\Vert
_{E}^{N}}{\int_{\mathbb{R}^{N}}\frac{\left\vert u\right\vert ^{N}}{\left\vert
x\right\vert ^{\beta}}dx}:u\in\bigskip E\setminus\left\{  0\right\}  \right\}
>0\text{ for any }0\leq\beta<N\text{.} \label{2.3}%
\end{equation}
In order to get the compactness of the embedding
\[
E\hookrightarrow L^{p}\left(  \mathbb{R}^{N}\right)  \text{ for all }p\geq N
\]
we also assume the following conditions on the potential $V$:

$(V2)$ $V(x)\rightarrow\infty$ as $\left\vert x\right\vert \rightarrow\infty$;
or more generally, for every $M>0$,
\[
\mu\left(  \left\{  x\in\mathbb{R}^{N}:V(x)\leq M\right\}  \right)  <\infty.
\]
or

$(V3)$ The function $\left[  V(x)\right]  ^{-1}$ belongs to $L^{1}\left(
\mathbb{R}^{N}\right)  $.

Now, from $(f1)$, we obtain for all $\left(  x,u\right)  \in\mathbb{R}%
^{N}\times%
\mathbb{R}
$,
\[
\left\vert F\left(  x,u\right)  \right\vert \leq b_{3}\left[  \exp\left(
\alpha_{1}\left\vert u\right\vert ^{N/(N-1)}\right)  -S_{N-2}\left(
\alpha_{1},u\right)  \right]
\]
for some constants $\alpha_{1}$, $b_{3}>0$. Thus, by Lemma 1.1, we have
$F\left(  x,u\right)  \in L^{1}\left(  \mathbb{R}^{N}\right)  $ for all $u\in
W^{1,N}\left(  \mathbb{R}^{N}\right)  $. Define the functionals
$J,~J_{\varepsilon}:E\rightarrow%
\mathbb{R}
$ by%

\begin{align*}
J_{\varepsilon}(u)  &  =\int_{\mathbb{R}^{N}}A(x,\nabla u)dx+\frac{1}{N}%
\int_{\mathbb{R}^{N}}V(x)\left\vert u\right\vert ^{N}dx-\int_{\mathbb{R}^{N}%
}\frac{F(x,u)}{\left\vert x\right\vert ^{\beta}}dx-\varepsilon\int
_{\mathbb{R}^{N}}hudx\\
J(u)  &  =\frac{1}{N}\int_{\mathbb{R}^{N}}\left\vert \nabla u\right\vert
^{N}dx+\frac{1}{N}\int_{\mathbb{R}^{N}}V(x)\left\vert u\right\vert ^{N}%
dx-\int_{\mathbb{R}^{N}}\frac{F(x,u)}{\left\vert x\right\vert ^{\beta}}dx
\end{align*}
then the functionals $J,~J_{\varepsilon}~$are well-defined by Lemma 1.1.
Moreover, $J,~J_{\varepsilon}~$are the $C^{1}$ functional on $E$ and $\forall
u,v\in E,$
\begin{align*}
DJ_{\varepsilon}\left(  u\right)  v  &  =\int_{\mathbb{R}^{N}}a\left(
x,\nabla u\right)  \nabla vdx+\int_{\mathbb{R}^{N}}V(x)\left\vert u\right\vert
^{N-2}vdx-\int_{\mathbb{R}^{N}}\frac{f(x,u)v}{\left\vert x\right\vert ^{\beta
}}dx-\varepsilon\int_{\mathbb{R}^{N}}hvdx\\
DJ\left(  u\right)  v  &  =\int_{\mathbb{R}^{N}}\left\vert \nabla u\right\vert
^{N-2}\nabla u\nabla vdx+\int_{\mathbb{R}^{N}}V(x)\left\vert u\right\vert
^{N-2}vdx-\int_{\mathbb{R}^{N}}\frac{f(x,u)v}{\left\vert x\right\vert ^{\beta
}}dx.~
\end{align*}
Note that in the case of $N-$Laplacian: $A\left(  x,\tau\right)  =\frac{1}%
{N}\left\vert \tau\right\vert ^{N}$, we choose
\[
a\left(  x,\tau\right)  =\left\vert \tau\right\vert ^{N-2}\tau,~k_{0}=\frac
{1}{N},~h_{1}\left(  x\right)  =1.
\]

We next state our main results.

\begin{theorem}
Suppose that (V1) and (V2) (or (V3)) and (f1)-(f2) are satisfied. Furthermore,
assume that
\[
\text{(f4) }\lim\sup_{s\rightarrow0+}\frac{F(x,s)}{k_{0}\left\vert
s\right\vert ^{N}}<\lambda_{1}(N)\text{ uniformly in }x\in\mathbb{R}^{N}.
\]
Then there exists $\varepsilon_{1}>0$ such that for each $0<\varepsilon
<\varepsilon_{1}$, problem (\ref{1.1}) has a nontrivial weak solution of
mountain-pass type.
\end{theorem}

\begin{theorem}
Suppose that (V1) and (V2) (or (V3)) and (f1)-(f3) are satisfied. Furthermore,
assume that
\[
\text{(f4) }\lim\sup_{s\rightarrow0+}\frac{NF(x,s)}{\left\vert s\right\vert
^{N}}<\lambda_{1}(N)\text{ uniformly in }x\in\mathbb{R}^{N}.
\]
and there exists $r>0$ such that
\begin{align*}
&  \text{(f5)}\underset{s\rightarrow\infty}{\lim}sf(x,s)\exp\left(
-\alpha_{0}\left\vert s\right\vert ^{N/(N-1)}\right) \\
&  >\frac{1}{\left[  \frac{r^{N-\beta}}{N-\beta}e^{(\alpha_{N}d(N-\beta
)/N)}+\mathcal{C}r^{N-\beta}-\frac{r^{N-\beta}}{N-\beta}\right]  }\left(
\frac{N-\beta}{\alpha_{0}}\right)  ^{N-1}>0
\end{align*}
uniformly on compact subsets of $\mathbb{R}^{N}$ where $d$ and $\mathcal{C}%
$\ will be defined in section 3. Then there exists $\varepsilon_{2}>0$, such
that for each $0<\varepsilon<\varepsilon_{2}$, problem (\ref{1.2}) has at
least two nontrivial weak solutions and one of them has a negative energy.
\end{theorem}

\begin{theorem}
Under the same hypotheses in Theorem 2.2, the problem without the perturbation
(\ref{1.3}) has a nontrivial weak solution.
\end{theorem}

As we remarked earlier in the introduction, all the main theorems above remain
to hold when the nonlinear term $f$ does not satisfy the Ambrosetti-Rabinowitz
condition. As a result, we then establish the existence and multiplicity of
solutions in a wider class of nonlinear terms. See Section 7 for more details.

\section{Preliminary Results}

First, we recall what we call the Radial Lemma (see \cite{BL, DoMS}) which
asserts:%
\[
\left\vert u(x)\right\vert ^{N}\leq\frac{N}{\omega_{N-1}}\frac{\left\Vert
u\right\Vert _{N}^{N}}{\left\vert x\right\vert ^{N}},\forall x\in
\mathbb{R}^{N}\setminus\left\{  0\right\}
\]
for all $u\in W^{1,N}\left(  \mathbb{R}^{N}\right)  $ radially
symmetric. Using this Radial Lemma, we can prove  the following two
lemmas (Lemmas \ref{lem3.1} and \ref{lem3.2}) with an easy
adaptation from Lemma 2.2 and Lemma 2.3 in \cite{DoMS} for $\beta=0$
and Lemma 4.2 in \cite{AY}.

\begin{lemma}\label{lem3.1}
For $\kappa>0,~0\leq\beta<N$ and $\left\Vert u\right\Vert _{E}\leq M$ with $M$
sufficiently small and $q>N$, we have
\[
\int_{\mathbb{R}^{N}}\frac{\left[  \exp\left(  \kappa\left\vert u\right\vert
^{N/(N-1)}\right)  -S_{N-2}\left(  \kappa,u\right)  \right]  \left\vert
u\right\vert ^{q}}{\left\vert x\right\vert ^{\beta}}dx\leq C\left(
N,\kappa\right)  \left\Vert u\right\Vert _{E}^{q}.
\]

\end{lemma}

\begin{lemma}\label{lem3.2}
Let $\kappa>0,~0\leq\beta<N,~u\in E$ and $\left\Vert u\right\Vert _{E}\leq M$
such that $M^{N/(N-1)}<\left(  1-\frac{\beta}{N}\right)  \frac{\alpha_{N}%
}{\kappa}$, then
\[
\int_{\mathbb{R}^{N}}\frac{\left[  \exp\left(  \kappa\left\vert u\right\vert
^{N/(N-1)}\right)  -S_{N-2}\left(  \kappa,u\right)  \right]  \left\vert
u\right\vert }{\left\vert x\right\vert ^{\beta}}dx\leq C\left(  N,M,\kappa
\right)  \left\Vert u\right\Vert _{p^{\prime}}%
\]
for some $p^{\prime}>N$.
\end{lemma}

Next, we have

\begin{lemma}
Let $\left\{  w_{k}\right\}  \subset W^{1,N}\left(  \Omega\right)  $ where
$\Omega$ is a bounded open set in $\mathbb{R}^{N},~\left\Vert \nabla
w_{k}\right\Vert _{L^{N}\left(  \Omega\right)  }\leq1$. If $w_{k}\rightarrow
w\neq0$ weakly and almost everywhere, $\nabla w_{k}\rightarrow\nabla w$ almost
everywhere, then $\frac{\exp\left\{  \alpha\left\vert w_{k}\right\vert
^{N/(N-1)}\right\}  }{\left\vert x\right\vert ^{\beta}}$ is bounded in
$L^{1}\left(  \Omega\right)  $ for $0<\alpha<\left(  1-\frac{\beta}{N}\right)
\alpha_{N}\left(  1-\left\Vert \nabla w\right\Vert _{L^{N}\left(
\Omega\right)  }^{N}\right)  ^{-1/(N-1)}.$
\end{lemma}

\begin{proof}
Using the Brezis-Lieb Lemma in \cite{BL}, we deduce that
\[
\left\Vert \nabla w_{k}\right\Vert _{L^{N}\left(  \Omega\right)  }%
^{N}-\left\Vert \nabla w_{k}-\nabla w\right\Vert _{L^{N}\left(  \Omega\right)
}^{N}\rightarrow\left\Vert \nabla w\right\Vert _{L^{N}\left(  \Omega\right)
}^{N}.
\]
Thus for $k$ large enough and $\delta>0$ small enough:
\[
0<\alpha\left(  1+\delta\right)  \left\Vert \nabla w_{k}-\nabla w\right\Vert
_{L^{N}\left(  \Omega\right)  }^{N/(N-1)}<\alpha_{N}\left(  1-\frac{\beta}%
{N}\right)  .
\]
By the singular Trudinger-Moser inequality on bounded domains \cite{AS}, we
get the conclusion.
\end{proof}

In the next two lemmas we check that the functional $J_{\varepsilon}$
satisfies the geometric conditions of the mountain-pass theorem. Then, we are
going to use a mountain-pass theorem without a compactness condition such as
the one of the (PS) type to prove the existence of the solution. This version
of the mountain-pass theorem is a consequence of the Ekeland's variational principle.

\begin{lemma}
Suppose that $(V1)$, $(f1)$ and $(f4)$ hold. Then there exists $\varepsilon
_{1}>0$ such that for $0<\varepsilon<\varepsilon_{1}$, there exists
$\rho_{\varepsilon}>0$ such that $J_{\varepsilon}(u)>0$ if $\left\Vert
u\right\Vert _{E}=\rho_{\varepsilon}$. Furthermore, $\rho_{\varepsilon}$ can
be chosen such that $\rho_{\varepsilon}\rightarrow0$ as $\varepsilon
\rightarrow0$.
\end{lemma}

\begin{proof}
From $(f4)$, there exist $\tau,~\delta>0$ such that $\left\vert u\right\vert
\leq\delta$ implies
\begin{equation}
F(x,u)\leq k_{0}\left(  \lambda_{1}\left(  N\right)  -\tau\right)  \left\vert
u\right\vert ^{N} \label{3.1}%
\end{equation}
for all $x\in\mathbb{R}^{N}$. Moreover, using $(f1)$ for each $q>N$, we can
find a constant $C=C(q,\delta)$ such that
\begin{equation}
F(x,u)\leq C\left\vert u\right\vert ^{q}\left[  \exp\left(  \kappa\left\vert
u\right\vert ^{N/(N-1)}\right)  -S_{N-2}\left(  \kappa,u\right)  \right]
\label{3.2}%
\end{equation}
for $\left\vert u\right\vert \geq\delta$ and $x\in\mathbb{R}^{N}$. From
(\ref{3.1}) and (\ref{3.2}) we have
\[
F(x,u)\leq k_{0}\left(  \lambda_{1}\left(  N\right)  -\tau\right)  \left\vert
u\right\vert ^{N}+C\left\vert u\right\vert ^{q}\left[  \exp\left(
\kappa\left\vert u\right\vert ^{N/(N-1)}\right)  -S_{N-2}\left(
\kappa,u\right)  \right]
\]
for all $\left(  x,u\right)  \in\mathbb{R}^{N}\times%
\mathbb{R}
$. Now, by $(A4)$, Lemma 3.2, (\ref{2.3}) and the continuous embedding
$E\hookrightarrow L^{N}\left(  \mathbb{R}^{N}\right)  $, we obtain
\begin{align*}
J_{\varepsilon}(u)  &  \geq k_{0}\left\Vert u\right\Vert _{E}^{N}-k_{0}\left(
\lambda_{1}\left(  N\right)  -\tau\right)  \int_{\mathbb{R}^{N}}%
\frac{\left\vert u\right\vert ^{N}}{\left\vert x\right\vert ^{\beta}%
}dx-C\left\Vert u\right\Vert _{E}^{q}-\varepsilon\left\Vert h\right\Vert
_{\ast}\left\Vert u\right\Vert _{E}\\
&  \geq k_{0}\left(  1-\frac{\left(  \lambda_{1}\left(  N\right)
-\tau\right)  }{\lambda_{1}\left(  N\right)  }\right)  \left\Vert u\right\Vert
_{E}^{N}-C\left\Vert u\right\Vert _{E}^{q}-\varepsilon\left\Vert h\right\Vert
_{\ast}\left\Vert u\right\Vert _{E}%
\end{align*}
Thus
\begin{equation}
J_{\varepsilon}(u)\geq\left\Vert u\right\Vert _{E}\left[  k_{0}\left(
1-\frac{\left(  \lambda_{1}\left(  N\right)  -\tau\right)  }{\lambda
_{1}\left(  N\right)  }\right)  \left\Vert u\right\Vert _{E}^{N-1}-C\left\Vert
u\right\Vert _{E}^{q-1}-\varepsilon\left\Vert h\right\Vert _{\ast}\right]
\label{3.3}%
\end{equation}
Since $\tau>0$ and $q>N$, we may choose $\rho>0$ such that $k_{0}\left(
1-\frac{\left(  \lambda_{1}\left(  N\right)  -\tau\right)  }{\lambda
_{1}\left(  N\right)  }\right)  \rho^{N-1}-C\rho^{q-1}>0$. Thus, if
$\varepsilon$ is sufficiently small then we can find some $\rho_{\varepsilon
}>0$ such that $J_{\varepsilon}(u)>0$ if $\left\Vert u\right\Vert
=\rho_{\varepsilon}$ and even $\rho_{\varepsilon}\rightarrow0$ as
$\varepsilon\rightarrow0$.
\end{proof}

\begin{lemma}
There exists $e\in E$ with $\left\Vert e\right\Vert _{E}>\rho_{\varepsilon} $
such that $J_{\varepsilon}(e)<\underset{\left\Vert u\right\Vert =\rho
_{\varepsilon}}{\inf}J_{\varepsilon}(u)$.
\end{lemma}

\begin{proof}
Let $u\in E\setminus\left\{  0\right\}  ,~u\geq0$ with compact support
$\Omega=supp(u)$. By $(f2),$ we have that for $p>N$, there exists a positive
constant $C>0$ such that
\begin{equation}
\forall s\geq0,~\forall x\in\Omega:~F\left(  x,s\right)  \geq cs^{p}-d.
\label{3.4}%
\end{equation}
Then by (\ref{2.2}), we get
\[
J_{\varepsilon}(tu)\leq Ct\int_{\Omega}h_{0}\left(  x\right)  \left\vert
\nabla u\right\vert dx+Ct^{N}\left\Vert u\right\Vert _{E}^{N}-Ct^{p}%
\int_{\Omega}\frac{\left\vert u\right\vert ^{p}}{\left\vert x\right\vert
^{\beta}}dx+C+\varepsilon t\left\vert \int_{\Omega}hudx\right\vert
\]
Since $p>N$, we have $J_{\varepsilon}(tu)\rightarrow-\infty$ as $t\rightarrow
\infty$. Setting $e=tu$ with $t$ sufficiently large, we get the conclusion.
\end{proof}

Now, we define the Moser Functions which have been frequently used in the
literature (see, for example, \cite{Do1}, \cite{DoMS}, \cite{AY}):%

\[
\widetilde{m}_{l}(x,r)=\frac{1}{\omega_{N-1}^{1/N}}\left\{
\begin{array}
[c]{l}%
\left(  \log l\right)  ^{(N-1)/N}\text{ \ \ \ if }|x|\leq\frac{r}{l}\\
\frac{\log\frac{r}{|x|}}{\left(  \log l\right)  ^{1/N}}\text{
\ \ \ \ \ \ \ \ \ \ \ if }\frac{r}{l}\leq|x|\leq r\\
0\text{ \ \ \ \ \ \ \ \ \ \ \ \ \ \ \ \ \ \ \ \ \ if }|x|\geq r
\end{array}
\right.  \bigskip
\]

We then immediately have $\widetilde{m}_{l}\left(  .,r\right)  \in
W^{1,N}(\mathbb{R}^{N})$, the support of $\widetilde{m}_{l}(x,r)$ is the ball
$B_{r}$, and
\begin{equation}
\int_{\mathbb{R}^{N}}\left\vert \nabla\widetilde{m}_{l}(x,r)\right\vert
^{N}dx=1,\text{ and }\left\Vert \widetilde{m}_{l}\right\Vert _{W^{1,N}%
(\mathbb{R}^{N})}^{N}=1+\frac{1}{\log l}\left(  \frac{\left(  N-1\right)
!}{N^{N}}r^{N}+o_{l}(1)\right)  . \label{5.1}%
\end{equation}
Then%
\[
\left\Vert \widetilde{m}_{l}\right\Vert _{E}^{N}\leq1+\frac{\underset
{\left\vert x\right\vert \leq r}{\max}V(x)}{\log l}\left(  \frac{\left(
N-1\right)  !}{N^{N}}r^{N}+o_{l}(1)\right)  .
\]
Consider $m_{l}(x,r)=\widetilde{m}_{l}(x,r)/\left\Vert \widetilde{m}%
_{l}\right\Vert _{E}$, we can write
\begin{equation}
m_{l}^{N/(N-1)}\left(  x,r\right)  =\omega_{N-1}^{-1/(N-1)}\log l+d_{l}\text{
for }\left\vert x\right\vert \leq r/l, \label{5.2}%
\end{equation}
Using (\ref{5.1}), we conclude that $\left\Vert \widetilde{m}_{l}\right\Vert
\rightarrow1$ as $l\rightarrow\infty$. Consequently,%
\begin{align}
\frac{d_{l}}{\log l}  &  \rightarrow0\text{ as }l\rightarrow\infty
,\label{5.3}\\
d  &  =\underset{l\rightarrow\infty}{\lim\inf}d_{l}\nonumber\\
d  &  \geq-\underset{\left\vert x\right\vert \leq r}{\max}V(x)\omega
_{N-1}^{-1/(N-1)}\frac{\left(  N-2\right)  !}{N^{N}}r^{N}.\nonumber
\end{align}

The following lemma was established in \cite{DoMS} when $\beta=0$.
 We adapt the proof given in \cite{DoMS} to our case $0\le \beta<N$. See also \cite{LaLu2} for a similar result on the Heisenberg group.

\begin{lemma}
\bigskip Suppose that (V1) and (f1)-(f5) hold. Then there exists $k\in%
\mathbb{N}
$ such that
\[
\underset{t\geq0}{\max}\left\{  \frac{t^{N}}{N}-\int_{\mathbb{R}^{N}}%
\frac{F\left(  x,tm_{k}\right)  }{\left\vert x\right\vert ^{\beta}}dx\right\}
<\frac{1}{N}\left(  \frac{N-\beta}{N}\frac{\alpha_{N}}{\alpha_{0}}\right)
^{N-1}%
\]

\end{lemma}

\begin{proof}
Choose $r>0$ as in the assumption $(f5)$ and $\beta_{0}>0$ such that
\begin{align}
\underset{s\rightarrow\infty}{\lim}sf(x,s)\exp\left(  -\alpha_{0}\left\vert
s\right\vert ^{N/(N-1)}\right)   &  \geq\beta_{0}\label{5.4}\\
&  >\frac{1}{\left[  \frac{r^{N-\beta}}{N-\beta}e^{(\alpha_{N}d(N-\beta
)/N)}+Cr^{N-\beta}-\frac{r^{N-\beta}}{N-\beta}\right]  }\left(  \frac{N-\beta
}{\alpha_{0}}\right)  ^{N-1},\nonumber
\end{align}
where
\begin{align*}
\mathcal{C}  &  =\underset{k\rightarrow\infty}{\lim}\zeta_{k}\log
k\int\limits_{0}^{\zeta_{k}^{-1}}\exp\left[  \left(  N-\beta\right)  \log
k\left(  s^{N/(N-1)}-\zeta_{k}s\right)  \right]  ds>0,~\zeta_{k}=\left\Vert
\widetilde{m}_{k}\right\Vert ,\\
\mathcal{C}  &  \geq\frac{1-e^{-(N-\beta)\log n}}{N-\beta}.
\end{align*}
Suppose, by contradiction, that for all $k$ we get
\[
\underset{t\geq0}{\max}\left\{  \frac{t^{N}}{N}-\int_{\mathbb{R}^{N}}%
\frac{F\left(  x,tm_{k}\right)  }{\left\vert x\right\vert ^{\beta}}dx\right\}
\geq\frac{1}{N}\left(  \frac{N-\beta}{N}\frac{\alpha_{N}}{\alpha_{0}}\right)
^{N-1}%
\]
where $m_{k}(x)=m_{k}(x,r)$. By (\ref{3.4}), for each $k$ there exists
$t_{k}>0$ such that
\[
\frac{t_{k}^{N}}{N}-\int_{\mathbb{R}^{N}}\frac{F\left(  x,t_{k}m_{k}\right)
}{\left\vert x\right\vert ^{\beta}}dx=\underset{t\geq0}{\max}\left\{
\frac{t^{N}}{N}-\int_{\mathbb{R}^{N}}\frac{F\left(  x,tm_{k}\right)
}{\left\vert x\right\vert ^{\beta}}dx\right\}
\]
Thus
\[
\frac{t_{k}^{N}}{N}-\int_{\mathbb{R}^{N}}\frac{F\left(  x,t_{k}m_{k}\right)
}{\left\vert x\right\vert ^{\beta}}dx\geq\frac{1}{N}\left(  \frac{N-\beta}%
{N}\frac{\alpha_{N}}{\alpha_{0}}\right)  ^{N-1}.
\]
From $F(x,u)\geq0$, we obtain
\begin{equation}
t_{k}^{N}\geq\left(  \frac{N-\beta}{N}\frac{\alpha_{N}}{\alpha_{0}}\right)
^{N-1} \label{5.5}%
\end{equation}
Since at $t=t_{k}$ we have
\[
\frac{d}{dt}\left(  \frac{t^{N}}{N}-\int_{\mathbb{R}^{N}}\frac{F\left(
x,tm_{k}\right)  }{\left\vert x\right\vert ^{\beta}}dx\right)  =0
\]
it follows that
\begin{equation}
t_{k}^{N}=\int_{\mathbb{R}^{N}}t_{k}m_{k}\frac{f\left(  x,t_{k}m_{k}\right)
}{\left\vert x\right\vert ^{\beta}}dx=\int_{\left\vert x\right\vert \leq
r}t_{k}m_{k}\frac{f\left(  x,t_{k}m_{k}\right)  }{\left\vert x\right\vert
^{\beta}}dx \label{5.6}%
\end{equation}
Using hypothesis $(f5)$, given $\tau>0$ there exists $R_{\tau}>0$ such that
for all $u\geq R_{\tau}$ and $\left\vert x\right\vert \leq r$, we have
\begin{equation}
uf(x,u)\geq\left(  \beta_{0}-\tau\right)  \exp\left(  \alpha_{0}\left\vert
u\right\vert ^{N/(N-1)}\right)  . \label{5.7}%
\end{equation}
From (\ref{5.6}) and (\ref{5.7}), for large $k$, we obtain
\begin{align*}
t_{k}^{N}  &  \geq\left(  \beta_{0}-\tau\right)  \int_{\left\vert x\right\vert
\leq\frac{r}{k}}\frac{\exp\left(  \alpha_{0}\left\vert t_{k}m_{k}\right\vert
^{N/(N-1)}\right)  }{\left\vert x\right\vert ^{\beta}}dx\\
&  =\left(  \beta_{0}-\tau\right)  \frac{\omega_{N-1}}{N-\beta}\left(
\frac{r}{k}\right)  ^{N-\beta}\exp\left(  \alpha_{0}t_{k}^{N/(N-1)}%
\omega_{N-1}^{-1/(N-1)}\log k+\alpha_{0}t_{k}^{N/(N-1)}d_{k}\right)
\end{align*}
Thus, setting
\[
L_{k}=\frac{\alpha_{0}N\log k}{\alpha_{N}}t_{k}^{N/(N-1)}+\alpha_{0}%
t_{k}^{N/(N-1)}d_{k}-N\log t_{k}-\left(  N-\beta\right)  \log k
\]
we have
\[
1\geq\left(  \beta_{0}-\tau\right)  \frac{\omega_{N-1}}{N-\beta}r^{N-\beta
}\exp L_{k}%
\]
Consequently, the sequence $\left(  t_{k}\right)  $ is bounded. Otherwise, up
to subsequences, we would have $\underset{k\rightarrow\infty}{\lim}%
L_{k}=\infty$ which leads to a contradiction. Moreover, by (\ref{5.3}),
(\ref{5.5}) and
\[
t_{k}^{N}\geq\left(  \beta_{0}-\tau\right)  \frac{\omega_{N-1}}{N-\beta
}r^{N-\beta}\exp\left[  \left(  N\frac{\alpha_{0}t_{k}^{N/(N-1)}}{\alpha_{N}%
}-\left(  N-\beta\right)  \right)  \log k+\alpha_{0}t_{k}^{N/(N-1)}%
d_{k}\right]
\]
it follows that
\begin{equation}
t_{k}^{N}\overset{k\rightarrow\infty}{\rightarrow}\left(  \frac{N-\beta}%
{N}\frac{\alpha_{N}}{\alpha_{0}}\right)  ^{N-1} \label{5.8}%
\end{equation}
Setting
\[
A_{k}=\left\{  x\in B_{r}:t_{k}m_{k}\geq R_{\tau}\right\}  \text{ and }%
B_{k}=B_{r}\setminus A_{k}%
\]
From (\ref{5.6}) and (\ref{5.7}) we have
\begin{align}
t_{k}^{N}  &  \geq\left(  \beta_{0}-\tau\right)  \int_{\left\vert x\right\vert
\leq r}\frac{\exp\left(  \alpha_{0}\left\vert t_{k}m_{k}\right\vert
^{N/(N-1)}\right)  }{\left\vert x\right\vert ^{\beta}}dx+\int_{B_{k}}%
\frac{t_{k}m_{k}f\left(  x,t_{k}m_{k}\right)  }{\left\vert x\right\vert
^{\beta}}dx\nonumber\\
&  -\left(  \beta_{0}-\tau\right)  \int_{B_{k}}\frac{\exp\left(  \alpha
_{0}\left\vert t_{k}m_{k}\right\vert ^{N/(N-1)}\right)  }{\left\vert
x\right\vert ^{\beta}}dx \label{5.9}%
\end{align}
Notice that $m_{k}(x)\rightarrow0$ and the characteristic functions
$\chi_{B_{k}}\rightarrow1$ for almost everywhere $x$ in $B_{r}$. Therefore the
Lebesgue's dominated convergence theorem implies%
\[
\int_{B_{k}}\frac{t_{k}m_{k}f\left(  x,t_{k}m_{k}\right)  }{\left\vert
x\right\vert ^{\beta}}dx\rightarrow0
\]
and
\[
\int_{B_{k}}\frac{\exp\left(  \alpha_{0}\left\vert t_{k}m_{k}\right\vert
^{N/(N-1)}\right)  }{\left\vert x\right\vert ^{\beta}}dx\rightarrow
\frac{\omega_{N-1}}{N-\beta}r^{N-\beta}%
\]
Moreover, using that
\[
t_{k}^{N}\overset{k\rightarrow\infty}{\underset{\geq}{\rightarrow}}\left(
\frac{N-\beta}{N}\frac{\alpha_{N}}{\alpha_{0}}\right)  ^{N-1}%
\]
we have
\begin{align*}
&  \int_{\left\vert x\right\vert \leq r}\frac{\exp\left(  \alpha_{0}\left\vert
t_{k}m_{k}\right\vert ^{N/(N-1)}\right)  }{\left\vert x\right\vert ^{\beta}%
}dx\\
&  \geq\int_{\left\vert x\right\vert \leq r}\frac{\exp\left(  \alpha
_{N}\left\vert m_{k}\right\vert ^{N/(N-1)}(N-\beta)/N\right)  }{\left\vert
x\right\vert ^{\beta}}dx\\
&  =\int_{\left\vert x\right\vert \leq r/k}\frac{\exp\left(  \alpha
_{N}\left\vert m_{k}\right\vert ^{N/(N-1)}(N-\beta)/N\right)  }{\left\vert
x\right\vert ^{\beta}}dx\\
&  +\int_{r/k\leq\left\vert x\right\vert \leq r}\frac{\exp\left(  \alpha
_{N}\left\vert m_{k}\right\vert ^{N/(N-1)}(N-\beta)/N\right)  }{\left\vert
x\right\vert ^{\beta}}dx
\end{align*}
and
\begin{align*}
&  \int_{\left\vert x\right\vert \leq r/k}\frac{\exp\left(  \alpha
_{N}\left\vert m_{k}\right\vert ^{N/(N-1)}(N-\beta)/N\right)  }{\left\vert
x\right\vert ^{\beta}}dx\\
&  =\int_{\left\vert x\right\vert \leq r/k}\frac{\exp\left[  \alpha_{N}%
\omega_{N-1}^{-1/(N-1)}\log k(N-\beta)/N+d_{k}\alpha_{N}(N-\beta)/N\right]
}{\left\vert x\right\vert ^{\beta}}dx\\
&  =\frac{\omega_{N-1}}{N-\beta}\left(  \frac{r}{k}\right)  ^{N-\beta
}k^{(N-\beta+\alpha_{N}\frac{d_{k}}{\log k}(N-\beta)/N)}\\
&  =\frac{\omega_{N-1}}{N-\beta}r^{N-\beta}k^{(\alpha_{N}\frac{d_{k}}{\log
k}(N-\beta)/N)}.
\end{align*}
Now, using the change of variable
\[
x=\frac{\log\left(  \frac{r}{s}\right)  }{\zeta_{k}\log k}\text{ with }%
\zeta_{k}=\left\Vert \widetilde{m}_{k}\right\Vert
\]
by straightforward computation, we have
\begin{align*}
&  \int_{r/k\leq\left\vert x\right\vert \leq r}\frac{\exp\left(  \alpha
_{N}\left\vert m_{k}\right\vert ^{N/(N-1)}(N-\beta)/N\right)  }{\left\vert
x\right\vert ^{\beta}}dx\\
&  =\omega_{N-1}r^{N-\beta}\zeta_{k}\log k\int\limits_{0}^{\zeta_{k}^{-1}}%
\exp\left[  \left(  N-\beta\right)  \log k\left(  s^{N/(N-1)}-\zeta
_{k}s\right)  \right]  ds
\end{align*}
which converges to $C\omega_{N-1}r^{N-\beta}$ as $k\rightarrow\infty$ where
\[
C=\underset{k\rightarrow\infty}{\lim}\zeta_{k}\log k\int\limits_{0}^{\zeta
_{k}^{-1}}\exp\left[  \left(  N-\beta\right)  \log k\left(  s^{N/(N-1)}%
-\zeta_{k}s\right)  \right]  ds>0.
\]
Finally, taking $k\rightarrow\infty$ in (\ref{5.9}), using (\ref{5.8}) and
using (\ref{5.3}) (see \cite{Do1, DoMS}), we obtain
\[
\left(  \frac{N-\beta}{N}\frac{\alpha_{N}}{\alpha_{0}}\right)  ^{N-1}%
\geq\left(  \beta_{0}-\tau\right)  \left[  \frac{\omega_{N-1}}{N-\beta
}r^{N-\beta}e^{(\alpha_{N}d(N-\beta)/N)}+C\omega_{N-1}r^{N-\beta}-\frac
{\omega_{N-1}}{N-\beta}r^{N-\beta}\right]
\]
which implies that
\[
\beta_{0}\leq\frac{1}{\left[  \frac{r^{N-\beta}}{N-\beta}e^{(\alpha
_{N}d(N-\beta)/N)}+Cr^{N-\beta}-\frac{r^{N-\beta}}{N-\beta}\right]  }\left(
\frac{N-\beta}{\alpha_{0}}\right)  ^{N-1}.
\]
This contradicts to (\ref{5.4}), and the proof is complete.
\end{proof}

\section{The existence of solution for the problem (\ref{1.1})}

It is well known that the failure of the (PS) compactness condition creates
difficulties in studying this class of elliptic problems involving critical
growth and unbounded domains. In next several lemmas, instead of (PS)
sequence, we will use and analyze the compactness of Cerami sequences of
$J_{\varepsilon}$.

\begin{lemma}
\bigskip Let $\left(  u_{k}\right)  \subset E$ be an arbitrary Cerami sequence
of $J_{\varepsilon}$, i.e.,
\[
J_{\varepsilon}\left(  u_{k}\right)  \rightarrow c,~\left(  1+\left\Vert
u_{k}\right\Vert _{E}\right)  \left\Vert DJ_{\varepsilon}\left(  u_{k}\right)
\right\Vert _{E^{\prime}}\rightarrow0\text{ as }k\rightarrow\infty.
\]
Then there exists a subsequence of $\left(  u_{k}\right)  $ (still denoted by
$\left(  u_{k}\right)  $) and $u\in E$ such that
\[
\left\{
\begin{array}
[c]{l}%
\frac{f(x,u_{k})}{\left\vert x\right\vert ^{\beta}}\rightarrow\frac
{f(x,u)}{\left\vert x\right\vert ^{\beta}}\text{
\ \ \ \ \ \ \ \ \ \ \ \ \ \ \ \ \ \ \ \ \ \ \ \ \ \ \ \ strongly in }%
L_{loc}^{1}\left(  \mathbb{R}^{N}\right) \\
\nabla u_{k}(x)\rightarrow\nabla u(x)\text{
\ \ \ \ \ \ \ \ \ \ \ \ \ \ \ \ \ \ \ \ \ \ \ almost everywhere in }%
\mathbb{R}^{N}\\
a\left(  x,\nabla u_{k}\right)  \rightharpoonup a\left(  x,\nabla u\right)
\text{ \ \ \ weakly in }\left(  L_{loc}^{N/(N-1)}\left(  \mathbb{R}%
^{N}\right)  \right)  ^{N}\\
u_{k}\rightharpoonup u\text{
\ \ \ \ \ \ \ \ \ \ \ \ \ \ \ \ \ \ \ \ \ \ \ \ \ \ \ \ \ \ \ \ \ \ \ \ \ \ \ \ \ \ \ weakly
in }E
\end{array}
\right.
\]
Furthermore $u$ is a weak solution of (\ref{1.1}).
\end{lemma}

For simplicity, we will only sketch the proof where includes the nonuniform
terms $a(x, \nabla u)$ and $A(x, \nabla u)$.

\begin{proof}
Let $v\in E$, then we have
\begin{equation}
\int_{\mathbb{R}^{N}}A(x,\nabla u_{k})dx+\frac{1}{N}\int_{\mathbb{R}^{N}%
}V(x)\left\vert u_{k}\right\vert ^{N}dx-\int_{\mathbb{R}^{N}}\frac{F(x,u_{k}%
)}{\left\vert x\right\vert ^{\beta}}dx-\varepsilon\int_{\mathbb{R}^{N}}%
hu_{k}dx\overset{k\rightarrow\infty}{\rightarrow}c \label{3.5}%
\end{equation}
and
\begin{align}
\left\vert DJ_{\varepsilon}\left(  u_{k}\right)  v\right\vert  &  =\left\vert
\int_{\mathbb{R}^{N}}a\left(  x,\nabla u_{k}\right)  \nabla vdx+\int
_{\mathbb{R}^{N}}V(x)\left\vert u_{k}\right\vert ^{N-2}u_{k}vdx-\int
_{\mathbb{R}^{N}}\frac{f(x,u_{k})v}{\left\vert x\right\vert ^{\beta}%
}dx-\varepsilon\int_{\mathbb{R}^{N}}hvdx\right\vert \nonumber\\
&  \leq\frac{\tau_{k}\left\Vert v\right\Vert _{E}}{\left(  1+\left\Vert
u_{k}\right\Vert _{E}\right)  } \label{3.6}%
\end{align}
where $\tau_{k}\rightarrow0$ as $k\rightarrow\infty$. Choosing $v=u_{k}$ in
(\ref{3.6}) and by $(A3)$, we get
\begin{align*}
&  \int_{\mathbb{R}^{N}}\frac{f(x,u_{k})u_{k}}{\left\vert x\right\vert
^{\beta}}dx+\varepsilon\int_{\mathbb{R}^{N}}hu_{k}dx-N\int_{\mathbb{R}^{N}%
}A\left(  x,\nabla u_{k}\right)  -\int_{\mathbb{R}^{N}}V(x)\left\vert
u_{k}\right\vert ^{N-2}u_{k}dx\\
&  \leq\tau_{k}\frac{\left\Vert u_{k}\right\Vert _{E}}{\left(  1+\left\Vert
u_{k}\right\Vert _{E}\right)  }\rightarrow0
\end{align*}
This together with (\ref{3.5}), $(f2)$ and $(A4)$ leads to
\[
\left(  \frac{p}{N}-1\right)  \left\Vert u_{k}\right\Vert _{E}^{N}\leq
C\left(  1+\left\Vert u_{k}\right\Vert _{E}\right)
\]
and hence $\left\Vert u_{k}\right\Vert _{E}$ is bounded and thus
\begin{equation}
\int_{\mathbb{R}^{N}}\frac{f(x,u_{k})u_{k}}{\left\vert x\right\vert ^{\beta}%
}dx\leq C,~\int_{\mathbb{R}^{N}}\frac{F(x,u_{k})}{\left\vert x\right\vert
^{\beta}}dx\leq C. \label{3.7}%
\end{equation}
Thanks to the assumptions on the potential $V$, the embedding
$E\hookrightarrow L^{q}\left(  \mathbb{R}^{N}\right)  $ is compact for all
$q\geq N,$ by extracting a subsequence, we can assume that
\[
u_{k}\rightarrow u\text{ weakly in~}E\text{ and for almost all }x\in
\mathbb{R}^{N}.
\]
Thanks to Lemma 2.1 in \cite{FMR}, we have%
\begin{equation}
\frac{f\left(  x,u_{n}\right)  }{\left\vert x\right\vert ^{\beta}}%
\rightarrow\frac{f\left(  x,u\right)  }{\left\vert x\right\vert ^{\beta}%
}\text{ in }L_{loc}^{1}\left(  \mathbb{R}^{N}\right)  . \label{3.8}%
\end{equation}

Next, up to a subsequence, we can define an energy concentration set for any
fixed $\delta>0$,
\[
\Sigma_{\delta}=\left\{  x\in\mathbb{R}^{N}:\underset{r\rightarrow0}{\lim
}\underset{k\rightarrow\infty}{\lim}\int_{B_{r}\left(  x\right)  }\left(
\left\vert u_{k}\right\vert ^{N}+\left\vert \nabla u_{k}\right\vert
^{N}\right)  dx^{\prime}\geq\delta\right\}
\]
Since $\left(  u_{k}\right)  $ is bounded, $\Sigma_{\delta}$ must be a finite
set. Adapting an argument similar to \cite{AY} (we omit the details here), we
can prove that for any compact set $K\subset\subset\mathbb{R}^{N}%
\setminus\Sigma_{\delta}$,
\begin{equation}
\underset{k\rightarrow\infty}{\lim}\int_{K}\frac{\left\vert f\left(
x,u_{k}\right)  u_{k}-f\left(  x,u\right)  u\right\vert }{\left\vert
x\right\vert ^{\beta}}dx=0 \label{3.15}%
\end{equation}
Next we will prove that for any compact set $K\subset\subset\mathbb{R}%
^{N}\setminus\Sigma_{\delta}$,%
\begin{equation}
\underset{k\rightarrow\infty}{\lim}\int_{K}\left\vert \nabla u_{k}-\nabla
u\right\vert ^{N}dx=0 \label{3.16}%
\end{equation}
It is enough to prove for any $x^{\ast}\in\mathbb{R}^{N}\setminus
\Sigma_{\delta}$, and $B_{r}(x^{\ast},r)\subset\mathbb{R}^{N}\setminus
\Sigma_{\delta}$, there holds
\begin{equation}
\underset{k\rightarrow\infty}{\lim}\int_{B_{r/2}\left(  x^{\ast}\right)
}\left\vert \nabla u_{k}-\nabla u\right\vert ^{N}dx=0 \label{3.17}%
\end{equation}
For this purpose, we take $\phi\in C_{0}^{\infty}\left(  B_{r}\left(  x^{\ast
}\right)  \right)  $ with $0\leq\phi\leq1$ and $\phi=1$ on $B_{r/2}\left(
x^{\ast}\right)  $. Obviously $\phi u_{k}$ is a bounded sequence. Choose
$h=\phi u_{k}$ and $h=\phi u$ in (\ref{3.6}), we have:
\begin{align*}
&  \int_{B_{r}\left(  x^{\ast}\right)  }\phi\left(  a\left(  x,\nabla
u_{k}\right)  -a\left(  x,\nabla u\right)  \right)  \left(  \nabla
u_{k}-\nabla u\right)  dx\leq\int_{B_{r}\left(  x^{\ast}\right)  }a\left(
x,\nabla u_{k}\right)  \nabla\phi\left(  u-u_{k}\right)  dx\\
&  +\int_{B_{r}\left(  x^{\ast}\right)  }\phi a\left(  x,\nabla u\right)
\left(  \nabla u-\nabla u_{k}\right)  dx+\int_{B_{r}\left(  x^{\ast}\right)
}\phi\left(  u_{k}-u\right)  \frac{f\left(  x,u_{k}\right)  }{\left\vert
x\right\vert ^{\beta}}dx\\
&  +\tau_{k}\left\Vert \phi u_{k}\right\Vert _{E}+\tau_{k}\left\Vert \phi
u\right\Vert _{E}-\varepsilon\int_{B_{r}\left(  x^{\ast}\right)  }\phi
h\left(  u_{k}-u\right)  dx
\end{align*}
Note that by Holder's inequality and the compact embedding of
$E\hookrightarrow L^{N}\left(  \Omega\right)  $, we get
\begin{equation}
\underset{k\rightarrow\infty}{\lim}\int_{B_{r}\left(  x^{\ast}\right)
}a\left(  x,\nabla u_{k}\right)  \nabla\phi\left(  u-u_{k}\right)  dx=0
\label{3.19}%
\end{equation}
Since $\nabla u_{k}\rightharpoonup\nabla u$ and $u_{k}\rightharpoonup u$,
there holds
\begin{equation}
\underset{k\rightarrow\infty}{\lim}\int_{B_{r}\left(  x^{\ast}\right)  }\phi
a\left(  x,\nabla u\right)  \left(  \nabla u-\nabla u_{k}\right)  dx=0\text{
and }\underset{k\rightarrow\infty}{\lim}\int_{B_{r}\left(  x^{\ast}\right)
}\phi h\left(  u_{k}-u\right)  dx=0 \label{3.20}%
\end{equation}
This implies that
\[
\underset{k\rightarrow\infty}{\lim}\int_{B_{r}\left(  x^{\ast}\right)  }%
\phi\left(  u_{k}-u\right)  f\left(  x,u_{k}\right)  dx=0
\]
So we can conclude that
\[
\underset{k\rightarrow\infty}{\lim}\int_{B_{r}\left(  x^{\ast}\right)  }%
\phi\left(  a\left(  x,\nabla u_{k}\right)  -a\left(  x,\nabla u\right)
\right)  \left(  \nabla u_{k}-\nabla u\right)  dx=0
\]
and hence we get (\ref{3.17}) by $(A2)$. Thus we have (\ref{3.16}) by a
covering argument. Since $\Sigma_{\delta}$ is finite, it follows that $\nabla
u_{k}$ converges to $\nabla u$ almost everywhere. This immediately implies, up
to a subsequence, $a\left(  x,\nabla u_{k}\right)  \rightharpoonup a\left(
x,\nabla u\right)  $ weakly in $\left(  L_{loc}^{N/(N-1)}\left(
\mathbb{R}^{N}\right)  \right)  ^{N-2}$. Using all these facts, letting $k$
tend to infinity in (\ref{3.6}) and combining with (\ref{3.8}), we obtain
\[
\left\langle DJ_{\varepsilon}(u),v\right\rangle =0~\forall v\in C_{0}^{\infty
}\left(  \mathbb{R}^{N}\right)  .
\]
This completes the proof of the Lemma.
\end{proof}

Now, we are ready to prove Theorem 2.1. The existence of the solution of
(\ref{1.1}) follows by a standard "mountain-pass" procedure.

\subsection{The proof of Theorem 2.1}

\begin{proposition}
Under the assumptions (V1) and (V2) (or (V3)), and (f1)-(f4), there exists
$\varepsilon_{1}>0$ such that for each $0<\varepsilon<\varepsilon_{1}$, the
problem (\ref{1.1}) has a solution $u_{M}$ via mountain-pass theorem.
\end{proposition}

\begin{proof}
For $\varepsilon$ sufficiently small, by Lemmas 3.4 and 3.5, $J_{\varepsilon}$
satisfies the hypotheses of the mountain-pass theorem except possibly for the
(PS) condition. Thus, using the mountain-pass theorem without the (PS)
condition, we can find a sequence $\left(  u_{k}\right)  $ in $E$ such that
\[
J_{\varepsilon}\left(  u_{k}\right)  \rightarrow c_{M}>0\text{ and }\left(
1+\left\Vert u_{k}\right\Vert _{E}\right)  \left\Vert DJ_{\varepsilon}\left(
u_{k}\right)  \right\Vert \rightarrow0
\]
where $c_{M}$ is the mountain-pass level of $J_{\varepsilon}$. Now, by Lemma
4.1, the sequence $\left(  u_{k}\right)  $ converges weakly to a weak solution
$u_{M}$ of (\ref{1.1}) in $E$. Moreover, $u_{M}\neq0$ since $h\neq0$.
\end{proof}

\section{The multiplicity results of the Problem (\ref{1.2})}

In this section, we deal with the problem (\ref{1.2}). Note that
this is the special case of the problem (\ref{1.1}) with $A\left(
x,\tau\right) =\frac{\left\vert \tau\right\vert ^{N}}{N}$. Some
preliminary lemmas in the case $\beta=0$ were treated in \cite{Do1,
DoMS}. We will not include details of the proof here, but refer the
reader to \cite{Do1, DoMS}. The key ingredient of this section is
the proof of Proposition \ref{prop5.2} which is substantially
different from those in \cite{Do1, DoMS}.

\begin{lemma}
There exists $\eta>0$ and $v\in E$ with $\left\Vert v\right\Vert _{E}=1$ such
that $J_{\varepsilon}(tv)<0$ for all $0<t<\eta$. In particular, $\underset
{\left\Vert u\right\Vert _{E}\leq\eta}{\inf}J_{\varepsilon}(u)<0 $.
\end{lemma}

\begin{corollary}
Under the hypotheses (V1) and (f1)-(f5), if $\varepsilon$ is sufficiently
small then
\[
\underset{t\geq0}{\max}J_{\varepsilon}\left(  tm_{k}\right)  =\underset
{t\geq0}{\max}\left\{  \frac{t^{N}}{N}-\int_{\mathbb{R}^{N}}\frac{F\left(
x,tm_{k}\right)  }{\left\vert x\right\vert ^{\beta}}dx-t\int_{\mathbb{R}^{N}%
}\varepsilon hm_{k}dx\right\}  <\frac{1}{N}\left(  \frac{N-\beta}{N}%
\frac{\alpha_{N}}{\alpha_{0}}\right)  ^{N-1}%
\]

\end{corollary}

Note that we can conclude by inequality (\ref{3.3}) and Lemma 5.1 that
\begin{equation}
-\infty<c_{0}=\underset{\left\Vert u\right\Vert _{E}\leq\rho_{\varepsilon}%
}{\inf}J_{\varepsilon}\left(  u\right)  <0. \label{5.10}%
\end{equation}
Next, we will prove that this infimum is achieved and generate a solution. In
order to obtain convergence results, we need to improve the estimate of Lemma 3.6.

\begin{corollary}
Under the hypotheses (V1) and (f1)-(f5), there exist $\varepsilon_{2}%
\in\left(  0,\varepsilon_{1}\right]  $ and $u\in W^{1,N}\left(  \mathbb{R}%
^{N}\right)  $ with compact support such that for all $0<\varepsilon
<\varepsilon_{2}$,
\[
J_{\varepsilon}\left(  tu\right)  <c_{0}+\frac{1}{N}\left(  \frac{N-\beta}%
{N}\frac{\alpha_{N}}{\alpha_{0}}\right)  ^{N-1}\text{ for all }t\geq0
\]

\end{corollary}

\begin{proof}
It is possible to increase the infimum $c_{0}$ by reducing $\varepsilon$. By
Lemma 3.4, $\rho_{\varepsilon}\overset{\varepsilon\rightarrow0}{\rightarrow}%
0$. Consequently, $c_{0}\overset{\varepsilon\rightarrow0}{\rightarrow}0$. Thus
there exists $\varepsilon_{2}>0$ such that if $0<\varepsilon<\varepsilon_{2}$
then, by Corollary 5.1, we have
\[
\underset{t\geq0}{\max}J_{\varepsilon}\left(  tm_{k}\right)  <c_{0}+\frac
{1}{N}\left(  \frac{N-\beta}{N}\frac{\alpha_{N}}{\alpha_{0}}\right)  ^{N-1}%
\]
Taking $u=m_{k}\in W^{1,N}\left(  \mathbb{R}^{N}\right)  $, the result follows.
\end{proof}

\begin{lemma}
If $\left(  u_{k}\right)  $ is a Cerami sequence for $J_{\varepsilon}$ at any
level with
\begin{equation}
\underset{k\rightarrow\infty}{\lim\inf}\left\Vert u_{k}\right\Vert
_{E}<\left(  \frac{N-\beta}{N}\frac{\alpha_{N}}{\alpha_{0}}\right)  ^{(N-1)/N}
\label{5.11}%
\end{equation}
then $\left(  u_{k}\right)  $ possesses a subsequence which converges strongly
to a solution $u_{0}$ of (\ref{1.2}).
\end{lemma}

\begin{proof}
See Lemma 5.2 in \cite{DoMS} for $\beta=0$ and Lemma 4.6 in
\cite{AY} for $0\le \beta<N$.
\end{proof}

\subsection{Proof of Theorem 2.2}

The proof of the existence of the second solution of (\ref{1.2}) follows by a
minimization argument and Ekeland's variational principle.

\begin{proposition}
There exists $\varepsilon_{2}>0$ such that for each $\varepsilon$ with
$0<\varepsilon<\varepsilon_{2}$, Eq. (\ref{1.2}) has a minimum type solution
$u_{0}$ with $J_{\varepsilon}\left(  u_{0}\right)  =c_{0}<0$, where $c_{0}$ is
defined in (\ref{5.10}).
\end{proposition}

\begin{proof}
Let $\rho_{\varepsilon}$ be as in Lemma 3.4. We can choose $\varepsilon_{2}>0$
sufficiently small such that
\[
\rho_{\varepsilon}<\left(  \frac{N-\beta}{N}\frac{\alpha_{N}}{\alpha_{0}%
}\right)  ^{(N-1)/N}%
\]
Since $\overline{B}_{\rho_{\varepsilon}}$ is a complete metric space with the
metric given by the norm of $E$, convex and the functional $J_{\varepsilon}$
is of class $C^{1}$ and bounded below on $\overline{B}_{\rho_{\varepsilon}}$,
by the Ekeland's variational principle there exists a sequence $\left(
u_{k}\right)  $ in $\overline{B}_{\rho_{\varepsilon}}$ such that
\[
J_{\varepsilon}\left(  u_{k}\right)  \rightarrow c_{0}=\underset{\left\Vert
u\right\Vert _{E}\leq\rho_{\varepsilon}}{\inf}J_{\varepsilon}\left(  u\right)
\text{ and }\left\Vert DJ_{\varepsilon}\left(  u_{k}\right)  \right\Vert
\rightarrow0
\]
Observing that
\[
\left\Vert u_{k}\right\Vert _{E}\leq\rho_{\varepsilon}<\left(  \frac{N-\beta
}{N}\frac{\alpha_{N}}{\alpha_{0}}\right)  ^{(N-1)/N}%
\]
by Lemma 5.2 it follows that there exists a subsequence of $\left(
u_{k}\right)  $ which converges to a solution $u_{0}$ of (\ref{1.2}).
Therefore, $J_{\varepsilon}\left(  u_{0}\right)  =c_{0}<0$.
\end{proof}

\begin{remark}
By Corollary 5.2, we can conclude that
\[
0<c_{M}<c_{0}+\frac{1}{N}\left(  \frac{N-\beta}{N}\frac{\alpha_{N}}{\alpha
_{0}}\right)  ^{N-1}%
\]

\end{remark}

\begin{proposition}\label{prop5.2}
If $\varepsilon_{2}>0$ is sufficiently small, then the solutions of
(\ref{1.3}) obtained in Propositions 4.1 and 5.1 are distinct.
\end{proposition}

\begin{remark}\label{rem5.2}

Before we  give a proof of the proposition, we like to make some remarks.
We note the following Hardy-Littlewood inequality holds for nonegative functions $f$ and $g$ in $\mathbb{R}^N$:
$$\int_{\mathbb{R}^N}f(x)g(x)dx\le \int_{\mathbb{R}^N} f^*(x)g^*(x)dx$$
where $f^*$ and $g^*$ are symmetric and decreasing rearrangement of $f$ and $g$ respectively.
However,
the following inequality:
$$\int_{|x|>R}f(x)dx\le \int_{|x|>R} f^*(x)dx$$
does not hold  for all $R>0$ in general.  Therefore, we will avoid using the symmetrization argument
when we prove \[
\int\limits_{\mathbb{R}^{N}}\frac{F(x,v_{k})}{\left\vert x\right\vert ^{\beta
}}dx\rightarrow\int\limits_{\mathbb{R}^{N}}\frac{F(x,u_{0})}{\left\vert
x\right\vert ^{\beta}}dx.
\]
Nevertheless, this can be taken care by a "double truncation" argument. This argument differs from those given in \cite{Do1, Do, DoMS, Y}.
Using this argument, the compact embedding $E\to L^q(\mathbb{R}^N)$ for $q\ge N$ is sufficient.

 \end{remark}

\begin{proof}
By Propositions 4.1 and 5.1, there exist sequences $\left(  u_{k}\right)  $,
$\left(  v_{k}\right)  $ in $E$ such that
\[
u_{k}\rightarrow u_{0},~J_{\varepsilon}\left(  u_{k}\right)  \rightarrow
c_{0}<0,~DJ_{\varepsilon}\left(  u_{k}\right)  u_{k}\rightarrow0
\]
and
\[
v_{k}\rightharpoonup u_{M},~J_{\varepsilon}\left(  v_{k}\right)  \rightarrow
c_{M}>0,~DJ_{\varepsilon}\left(  v_{k}\right)  v_{k}\rightarrow0,~\nabla
v_{k}(x)\rightarrow\nabla u_{M}(x)\text{ almost everywhere in }\mathbb{R}^{N}%
\]
Now, suppose by contradiction that $u_{0}=u_{M}$. As in the proof of Lemma 4.1
we obtain%
\begin{equation}
\frac{f(x,v_{k})}{\left\vert x\right\vert ^{\beta}}\rightarrow\frac
{f(x,u_{0})}{\left\vert x\right\vert ^{\beta}}\text{ in }L^{1}\left(
B_{R}\right)  \text{ for all }R>0 \label{5.14}%
\end{equation}
Moreover, by $(f2),~(f3)$
\[
\frac{F(x,v_{k})}{\left\vert x\right\vert ^{\beta}}\leq\frac{R_{0}f(x,v_{k}%
)}{\left\vert x\right\vert ^{\beta}}+\frac{M_{0}f(x,v_{k})}{\left\vert
x\right\vert ^{\beta}}%
\]
so by the Generalized Lebesgue's Dominated Convergence Theorem,%
\[
\frac{F(x,v_{k})}{\left\vert x\right\vert ^{\beta}}\rightarrow\frac
{F(x,u_{0})}{\left\vert x\right\vert ^{\beta}}\text{ in }L^{1}\left(
B_{R}\right)  .
\]
We will prove that%
\[
\int\limits_{\mathbb{R}^{N}}\frac{F(x,v_{k})}{\left\vert x\right\vert ^{\beta
}}dx\rightarrow\int\limits_{\mathbb{R}^{N}}\frac{F(x,u_{0})}{\left\vert
x\right\vert ^{\beta}}dx.
\]

It's sufficient to prove that given $\delta>0$, there exists $R>0$ such that
\[
\int\limits_{\left\vert x\right\vert >R}\frac{F(x,v_{k})}{\left\vert
x\right\vert ^{\beta}}dx\leq3\delta\text{ and }\int\limits_{\left\vert
x\right\vert >R}\frac{F(x,u_{0})}{\left\vert x\right\vert ^{\beta}}%
dx\leq3\delta.
\]

To prove it, we recall the following facts from our assumptions on
nonlinearity: there exists $c>0$ such that for all $\left(  x,s\right)
\in\mathbb{R}^{N}\times\mathbb{R}^{+}:$%
\begin{align}
F(x,s)  &  \leq c\left\vert s\right\vert ^{N}+cf(x,s)\label{5.15a}\\
F(x,s)  &  \leq c\left\vert s\right\vert ^{N}+cR\left(  \alpha_{0},s\right)
s\nonumber\\
\int_{\mathbb{R}^{N}}\frac{f(x,v_{k})v_{k}}{\left\vert x\right\vert ^{\beta}%
}dx  &  \leq C,~\int_{\mathbb{R}^{N}}\frac{F(x,v_{k})}{\left\vert x\right\vert
^{\beta}}dx\leq C.\nonumber
\end{align}

First, we will prove it for the case $\beta>0.$

We have that
\begin{align*}
\int\limits_{\substack{\left\vert x\right\vert >R\\\left\vert v_{k}\right\vert
>A}}\frac{F(x,v_{k})}{\left\vert x\right\vert ^{\beta}}dx &  \leq
c\int\limits_{\left\vert x\right\vert >R}\frac{\left\vert v_{k}\right\vert
^{N}}{\left\vert x\right\vert ^{\beta}}dx+c\int\limits_{\substack{\left\vert
x\right\vert >R\\\left\vert v_{k}\right\vert >A}}\frac{f(x,v_{k})}{\left\vert
x\right\vert ^{\beta}}dx\\
&  \leq\frac{c}{R^{\beta}}\left\Vert v_{k}\right\Vert _{E}^{N}+c\frac{1}%
{A}\int_{\mathbb{R}^{N}}\frac{f(x,v_{k})v_{k}}{\left\vert x\right\vert
^{\beta}}dx.
\end{align*}
Since $\left\Vert v_{k}\right\Vert _{E}$ is bounded and using (\ref{5.15a}),
we can first choose $A$ such that%
\[
c\frac{1}{A}\int_{\mathbb{R}^{N}}\frac{f(x,v_{k})v_{k}}{\left\vert
x\right\vert ^{\beta}}dx<\delta\text{~for all }k
\]
and then choose $R$ such that
\[
\frac{c}{R^{\beta}}\left\Vert v_{k}\right\Vert _{E}^{N}<\delta
\]
which thus%
\[
\int\limits_{\substack{\left\vert x\right\vert >R\\\left\vert v_{k}\right\vert
>A}}\frac{F(x,v_{k})}{\left\vert x\right\vert ^{\beta}}dx\leq2\delta.
\]
Now, note that with such $A$, we have for $\left\vert s\right\vert \leq A$:
\begin{align*}
F(x,s) &  \leq c\left\vert s\right\vert ^{N}+cR\left(  \alpha_{0},s\right)
s\\
&  \leq c\left\vert s\right\vert ^{N}+c%
{\displaystyle\sum\limits_{j=N-1}^{\infty}}
\frac{\alpha_{0}^{j}}{j!}\left\vert s\right\vert ^{Nj/(N-1)+1}\\
&  \leq\left\vert s\right\vert ^{N}\left[  c+c%
{\displaystyle\sum\limits_{j=N-1}^{\infty}}
\frac{\alpha_{0}^{j}}{j!}A^{Nj/(N-1)+1-N}\right]  \\
&  \leq C(\alpha_{0},A)\left\vert s\right\vert ^{N}.
\end{align*}
So we get%
\begin{align*}
\int\limits_{\substack{\left\vert x\right\vert >R\\\left\vert v_{k}\right\vert
\leq A}}\frac{F(x,v_{k})}{\left\vert x\right\vert ^{\beta}}dx &  \leq
\frac{C(\alpha_{0},A)}{R^{\beta}}\int\limits_{\substack{\left\vert
x\right\vert >R\\\left\vert v_{k}\right\vert \leq A}}\left\vert v_{k}%
\right\vert ^{N}dx\\
&  \leq\frac{C(\alpha_{0},A)}{R^{\beta}}\left\Vert v_{k}\right\Vert _{E}^{N}.
\end{align*}
Again, note that $\left\Vert v_{k}\right\Vert _{E}$ is bounded, we can choose
$R$ such that
\[
\int\limits_{\substack{\left\vert x\right\vert >R\\\left\vert v_{k}\right\vert
\leq A}}\frac{F(x,v_{k})}{\left\vert x\right\vert ^{\beta}}dx\leq\delta.
\]
In conclusion, we can choose $R>0$ such that
\[
\int\limits_{\left\vert x\right\vert >R}\frac{F(x,v_{k})}{\left\vert
x\right\vert ^{\beta}}dx\leq3\delta.
\]
Similarly, we can choose $R>0$ such that
\[
\int\limits_{\left\vert x\right\vert >R}\frac{F(x,u_{0})}{\left\vert
x\right\vert ^{\beta}}dx\leq3\delta.
\]

Now, if $\beta=0$, similarly, we have
\begin{align*}
\int\limits_{\substack{\left\vert x\right\vert >R\\\left\vert v_{k}\right\vert
>A}}F(x,v_{k})dx &  \leq c\int\limits_{\left\vert x\right\vert >R}\left\vert
v_{k}\right\vert ^{N}dx+c\int\limits_{\substack{\left\vert x\right\vert
>R\\\left\vert v_{k}\right\vert >A}}f(x,v_{k})dx\\
&  \leq\frac{c}{A}\int\limits_{\left\vert x\right\vert >R}\left\vert
v_{k}\right\vert ^{N+1}dx+c\frac{1}{A}\int_{\mathbb{R}^{N}}f(x,v_{k})v_{k}dx\\
&  \leq\frac{c}{A}\left\Vert v_{k}\right\Vert _{E}^{N+1}+c\frac{1}{A}%
\int_{\mathbb{R}^{N}}f(x,v_{k})v_{k}dx
\end{align*}
so since $\left\Vert v_{k}\right\Vert _{E}$ is bounded and by (\ref{5.15a}),
we can choose $A$ such that
\[
\int\limits_{\substack{\left\vert x\right\vert >R\\\left\vert v_{k}\right\vert
>A}}F(x,v_{k})dx\leq2\delta.
\]
Next, we have
\begin{align*}
\int\limits_{\substack{\left\vert x\right\vert >R\\\left\vert v_{k}\right\vert
\leq A}}F(x,v_{k})dx &  \leq C(\alpha_{0},A)\int\limits_{\substack{\left\vert
x\right\vert >R\\\left\vert v_{k}\right\vert \leq A}}\left\vert v_{k}%
\right\vert ^{N}dx\\
&  \leq2^{N-1}C(\alpha_{0},A)\left(  \int\limits_{\substack{\left\vert
x\right\vert >R\\\left\vert v_{k}\right\vert \leq A}}\left\vert v_{k}%
-u_{0}\right\vert ^{N}dx+\int\limits_{\substack{\left\vert x\right\vert
>R\\\left\vert v_{k}\right\vert \leq A}}\left\vert u_{0}\right\vert
^{N}dx\right)  .
\end{align*}
Now, using the compactness of embedding $E\hookrightarrow L^{q}\left(
\mathbb{R}^{N}\right)  ,~q\geq N$ and noticing that $v_{k}\rightharpoonup u_{0}$,
again we can choose $R$ such that
\[
\int\limits_{\substack{\left\vert x\right\vert >R\\\left\vert v_{k}\right\vert
\leq A}}F(x,v_{k})dx\leq\delta.
\]

Combining all the above estimates, we have the fact that
\begin{equation}\label{convergence}
\int\limits_{\mathbb{R}^{N}}\frac{F(x,v_{k})}{\left\vert x\right\vert ^{\beta
}}dx\rightarrow\int\limits_{\mathbb{R}^{N}}\frac{F(x,u_{0})}{\left\vert
x\right\vert ^{\beta}}dx.
\end{equation}
since $\delta$ is arbitrary and (\ref{5.14}) holds.

The remaining argument is similar to that in \cite{DoMS} when
$\beta=0$. For $0\le \beta<N$, we also refer to \cite{LaLu2}. For
completeness, we will include a proof here.

 From the above convergence formula (\ref{convergence}), we
have
\begin{equation}
\underset{k\rightarrow\infty}{\lim}\left\Vert \nabla v_{k}\right\Vert _{N}%
^{N}=Nc_{M}-\underset{k\rightarrow\infty}{\lim}\int_{\mathbb{R}^{N}%
}V(x)\left\vert v_{k}\right\vert ^{N}dx+N\int_{\mathbb{R}^{N}}\frac
{F(x,u_{0})}{\left\vert x\right\vert ^{\beta}}dx+N\varepsilon\int
_{\mathbb{R}^{N}}hu_{0}dx \label{5.16}%
\end{equation}
Now, let
\[
w_{k}=\frac{v_{k}}{\left\Vert \nabla v_{k}\right\Vert _{N}}\text{ and }%
w_{0}=\frac{u_{0}}{\lim_{k\rightarrow\infty}\left\Vert \nabla v_{k}\right\Vert
_{N}}%
\]
we have $\left\Vert \nabla w_{k}\right\Vert _{N}=1$ for all $k$ and
$w_{k}\rightharpoonup w_{0}$ in $D^{1,N}\left(  \mathbb{R}^{N}\right)  $, the
closure of the space $C_{0}^{\infty}\left(  \mathbb{R}^{N}\right)  $ endowed
with the norm $\left\Vert \nabla\varphi\right\Vert _{N}$. In particular,
$\left\Vert \nabla w_{0}\right\Vert _{N}\leq1$ and $w_{k}|_{B_{R}%
}\rightharpoonup w_{0}|_{B_{R}}$ in $W^{1,N}\left(  B_{R}\right)  $ for all
$R>0$. We claim that $\left\Vert \nabla w_{0}\right\Vert _{N}<1$.

Indeed, if $\left\Vert \nabla w_{0}\right\Vert _{N}=1$, then we have
$\lim_{k\rightarrow\infty}\left\Vert \nabla v_{k}\right\Vert _{N}=\left\Vert
\nabla u_{0}\right\Vert _{N}$ and thus $v_{k}\rightarrow u_{0}$ in
$W^{1,N}\left(  \mathbb{R}^{N}\right)  $ since $v_{k}\rightarrow u_{0}$ in
$L^{q}\left(  \mathbb{R}^{N}\right)  ,~q\geq N$. So we can find $g\in
W^{1,N}\left(  \mathbb{R}^{N}\right)  $ (for some $q\geq N$)$~$such that
$\left\vert v_{k}(x)\right\vert \leq g(x)$ almost everywhere in $\mathbb{R}%
^{N}$. From assumption $(f1)$, we have for some $\alpha_{1}>\alpha_{0}$ that
\[
\left\vert f(x,s)s\right\vert \leq b_{1}\left\vert s\right\vert ^{N}+C\left[
\exp\left(  \alpha_{1}\left\vert s\right\vert ^{N/(N-1)}\right)
-S_{N-2}\left(  \alpha_{1},s\right)  \right]
\]
for all $\left(  x,s\right)  \in\mathbb{R}^{N}\mathbb{\times R}$. Thus,
\begin{align*}
\frac{\left\vert f(x,v_{k})v_{k}\right\vert }{\left\vert x\right\vert ^{\beta
}}  &  \leq b_{1}\frac{\left\vert v_{k}\right\vert ^{N}}{\left\vert
x\right\vert ^{\beta}}+C\frac{\left[  \exp\left(  \alpha_{1}\left\vert
v_{k}\right\vert ^{N/(N-1)}\right)  -S_{N-2}\left(  \alpha_{1},v_{k}\right)
\right]  }{\left\vert x\right\vert ^{\beta}}\\
&  \leq b_{1}\frac{\left\vert v_{k}\right\vert ^{N}}{\left\vert x\right\vert
^{\beta}}+C\frac{\left[  \exp\left(  \alpha_{1}\left\vert g\right\vert
^{N/(N-1)}\right)  -S_{N-2}\left(  \alpha_{1},g\right)  \right]  }{\left\vert
x\right\vert ^{\beta}}%
\end{align*}
almost everywhere in $\mathbb{R}^{N}$. Now, by the  Lebesgue's
dominated convergence theorem,
\[
\underset{k\rightarrow\infty}{\lim}\int\limits_{\mathbb{R}^{N}}\frac
{f(x,v_{k})v_{k}}{\left\vert x\right\vert ^{\beta}}dx=\int\limits_{\mathbb{R}%
^{N}}\frac{f(x,u_{0})u_{0}}{\left\vert x\right\vert ^{\beta}}dx
\]
Similarly, since $u_{k}\rightarrow u_{0}$ in $E$, we also have
\[
\underset{k\rightarrow\infty}{\lim}\int\limits_{\mathbb{R}^{N}}\frac
{f(x,u_{k})u_{k}}{\left\vert x\right\vert ^{\beta}}dx=\int\limits_{\mathbb{R}%
^{N}}\frac{f(x,u_{0})u_{0}}{\left\vert x\right\vert ^{\beta}}dx
\]
Now, note that
\[
DJ_{\varepsilon}\left(  u_{k}\right)  u_{k}=\left\Vert u_{k}\right\Vert
_{E}^{N}-\int\limits_{\mathbb{R}^{N}}\frac{f(x,u_{k})u_{k}}{\left\vert
x\right\vert ^{\beta}}dx-\int\limits_{\mathbb{R}^{N}}\varepsilon
hu_{k}dx\rightarrow0
\]
and
\[
DJ_{\varepsilon}\left(  v_{k}\right)  v_{k}=\left\Vert v_{k}\right\Vert
_{E}^{N}-\int\limits_{\mathbb{R}^{N}}\frac{f(x,v_{k})v_{k}}{\left\vert
x\right\vert ^{\beta}}dx-\int\limits_{\mathbb{R}^{N}}\varepsilon
hv_{k}dx\rightarrow0
\]
we conclude that
\[
\underset{k\rightarrow\infty}{\lim}\left\Vert v_{k}\right\Vert _{E}%
^{N}=\underset{k\rightarrow\infty}{\lim}\left\Vert u_{k}\right\Vert _{E}%
^{N}=\left\Vert u_{0}\right\Vert _{E}^{N}%
\]
and thus $J_{\varepsilon}\left(  v_{k}\right)  \rightarrow J_{\varepsilon
}\left(  u_{0}\right)  =c_{0}<0$ and this is a contradiction.

So $\left\Vert \nabla w_{0}\right\Vert _{N}<1$. Using Remark 5.1 we have
\[
c_{M}-J_{\varepsilon}\left(  u_{0}\right)  <\frac{1}{N}\left(  \frac{N-\beta
}{N}\frac{\alpha_{N}}{\alpha_{0}}\right)  ^{N-1}%
\]
and thus
\[
\alpha_{0}<\frac{N-\beta}{N}\frac{\alpha_{N}}{\left[  N\left(  c_{M}%
-J_{\varepsilon}\left(  u_{0}\right)  \right)  \right]  ^{1/(N-1)}}%
\]
Now if we choose $q>1$ sufficiently close to 1 and set
\[
L(w)=c_{M}-\frac{1}{N}\int_{\mathbb{R}^{N}}V(x)\left\vert w\right\vert
^{N}dx+\int_{\mathbb{R}^{N}}\frac{F\left(  x,w\right)  }{\left\vert
x\right\vert ^{\beta}}dx+\varepsilon\int_{\mathbb{R}^{N}}hwdx
\]
then for some $\delta>0$,
\begin{align*}
q\alpha_{0}\left\Vert \nabla v_{k}\right\Vert _{N}^{N/(N-1)}  &  \leq
\frac{N-\beta}{N}\frac{\alpha_{N}\left\Vert \nabla v_{k}\right\Vert
_{N}^{N/(N-1)}}{\left[  N\left(  c_{M}-J_{\varepsilon}\left(  u_{0}\right)
\right)  \right]  ^{1/(N-1)}}-\delta\\
&  =\frac{N-\beta}{N}\frac{\alpha_{N}\left(  NL(v_{k})\right)  ^{1/(N-1)}%
+o_{k}(1)}{\left[  N\left(  c_{M}-J_{\varepsilon}\left(  u_{0}\right)
\right)  \right]  ^{1/(N-1)}}-\delta
\end{align*}
Note that
\[
\underset{k\rightarrow\infty}{\lim}L(v_{k})=c_{M}-\underset{k\rightarrow
\infty}{\lim}\frac{1}{N}\int_{\mathbb{R}^{N}}V(x)\left\vert v_{k}\right\vert
^{N}dx+\int_{\mathbb{R}^{N}}\frac{F\left(  x,u_{0}\right)  }{\left\vert
x\right\vert ^{\beta}}dx+\varepsilon\int_{\mathbb{R}^{N}}hu_{0}dx+o_{k}(1)
\]
and
\begin{align*}
&  \left(  c_{M}-\underset{k\rightarrow\infty}{\lim}\frac{1}{N}\int
_{\mathbb{R}^{N}}V(x)\left\vert v_{k}\right\vert ^{N}dx+\int_{\mathbb{R}^{N}%
}\frac{F\left(  x,u_{0}\right)  }{\left\vert x\right\vert ^{\beta}%
}dx+\varepsilon\int_{\mathbb{R}^{N}}hu_{0}dx\right)  \left(  1-\left\Vert
\nabla_{\mathbb{R}^{N}}w_{0}\right\Vert _{N}^{N}\right) \\
&  \leq c_{M}-J_{\varepsilon}\left(  u_{0}\right)
\end{align*}
so for $k,R$ sufficiently large,
\[
q\alpha_{0}\left\Vert \nabla v_{k}\right\Vert _{N}^{N/(N-1)}\leq\frac{N-\beta
}{N}\frac{\alpha_{N}}{\left[  1-\left\Vert \nabla_{\mathbb{R}^{N}}%
w_{0}\right\Vert _{L^{N}\left(  B_{R}\right)  }^{N}\right]  ^{1/(N-1)}}-\delta
\]
By Lemma 3.3, note that $\nabla w_{k}\rightarrow\nabla w_{0}$ almost
everywhere since $\nabla v_{k}(x)\rightarrow\nabla u_{M}(x)=\nabla u_{0}(x)$
almost everywhere in $\mathbb{R}^{N}$:
\begin{equation}
\int_{B_{R}}\frac{\exp\left(  q\alpha_{0}\left\Vert \nabla v_{k}\right\Vert
_{N}^{N/(N-1)}\left\vert w_{k}\right\vert ^{N/(N-1)}\right)  }{\left\vert
x\right\vert ^{\beta}}dx\leq C \label{5.17}%
\end{equation}

By $(f1)$ and Holder's inequality,
\begin{align*}
&  \left\vert \int_{\mathbb{R}^{N}}\frac{f\left(  x,v_{k}\right)  \left(
v_{k}-u_{0}\right)  }{\left\vert x\right\vert ^{\beta}}dx\right\vert \\
&  \leq b_{1}\int_{\mathbb{R}^{N}}\frac{\left\vert v_{k}\right\vert
^{N-1}\left\vert v_{k}-u_{0}\right\vert }{\left\vert x\right\vert ^{\beta}%
}dx+b_{2}\int_{B_{R}}\frac{\left\vert v_{k}-u_{0}\right\vert \exp\left(
\alpha_{0}\left\vert v_{k}\right\vert ^{N/(N-1)}\right)  }{\left\vert
x\right\vert ^{\beta}}dx\\
&  \leq b_{1}\left(  \int_{\mathbb{R}^{N}}\frac{\left\vert v_{k}\right\vert
^{N}}{\left\vert x\right\vert ^{\beta}}dx\right)  ^{(N-1)/N}\left(
\int_{\mathbb{R}^{N}}\frac{\left\vert v_{k}-u_{0}\right\vert ^{N}}{\left\vert
x\right\vert ^{\beta}}dx\right)  ^{1/N}\\
&  +b_{2}\left(  \int_{\mathbb{R}^{N}}\frac{\left\vert v_{k}-u_{0}\right\vert
^{q^{\prime}}}{\left\vert x\right\vert ^{\beta}}dx\right)  ^{1/q^{\prime}%
}\left(  \int_{B_{R}}\frac{\exp\left(  q\alpha_{0}\left\Vert \nabla
v_{k}\right\Vert _{N}^{N/(N-1)}\left\vert w_{k}\right\vert ^{N/(N-1)}\right)
}{\left\vert x\right\vert ^{\beta}}dx\right)  ^{1/q}%
\end{align*}
where $q^{\prime}=q/(q-1)$. By (\ref{5.17}), we have
\[
\left\vert \int_{\mathbb{R}^{N}}\frac{f\left(  x,v_{k}\right)  \left(
v_{k}-u_{0}\right)  }{\left\vert x\right\vert ^{\beta}}dx\right\vert \leq
C_{1}\left\Vert \frac{v_{k}-u_{0}}{\left\vert x\right\vert ^{\beta/N}%
}\right\Vert _{N}+C_{2}\left\Vert \frac{v_{k}-u_{0}}{\left\vert x\right\vert
^{\beta/q^{\prime}}}\right\Vert _{q^{\prime}}.
\]
Using the Holder inequality and the compact embedding $E\hookrightarrow
L^{q},~q\geq N$, we get%

\begin{align*}
\int_{\mathbb{R}^{N}}\frac{\left\vert v_{k}-u_{0}\right\vert ^{N}}{\left\vert
x\right\vert ^{\beta}}dx  & =\int_{\left\vert x\right\vert <1}\frac{\left\vert
v_{k}-u_{0}\right\vert ^{N}}{\left\vert x\right\vert ^{\beta}}dx+\int
_{\left\vert x\right\vert \geq1}\frac{\left\vert v_{k}-u_{0}\right\vert ^{N}%
}{\left\vert x\right\vert ^{\beta}}dx\\
& \leq\left(  \int_{\left\vert x\right\vert <1}\frac{1}{\left\vert
x\right\vert ^{\beta s}}dx\right)  ^{1/s}\left(  \int_{\left\vert x\right\vert
<1}\left\vert v_{k}-u_{0}\right\vert ^{s^{\prime}N}dx\right)  ^{1/s^{\prime}%
}+\left\Vert v_{k}-u_{0}\right\Vert _{N}^{N}\\
& \rightarrow0\text{ as }k\rightarrow\infty
\end{align*}
for some $s>1$ sufficiently close to 1.
Similarly,
\[
\int_{\mathbb{R}^{N}}\frac{\left\vert v_{k}-u_{0}\right\vert ^{q^{\prime}}%
}{\left\vert x\right\vert ^{\beta}}dx\overset{k\rightarrow\infty}{\rightarrow
}0.
\]
Thus we can conclude that
\[
\int_{\mathbb{R}^{N}}\left\vert \nabla v_{k}\right\vert ^{N-2}\nabla
v_{k}\left(  \nabla v_{k}-\nabla u_{0}\right)  dx+\int_{\mathbb{R}^{N}%
}V(x)\left\vert v_{k}\right\vert ^{N-2}v_{k}\left(  v_{k}-u_{0}\right)
dx\rightarrow0
\]
since $DJ_{\varepsilon}\left(  v_{k}\right)  \left(  v_{k}-u_{0}\right)
\rightarrow0$.

On the other hand, since $v_{k}\rightharpoonup u_{0}$
\[
\int_{\mathbb{R}^{N}}\left\vert \nabla u_{0}\right\vert ^{N-2}\nabla
u_{0}\left(  \nabla v_{k}-\nabla u_{0}\right)  dx\rightarrow0
\]
and
\[
\int_{\mathbb{R}^{N}}V(x)\left\vert u_{0}\right\vert ^{N-2}u_{0}\left(
v_{k}-u_{0}\right)  dx\rightarrow0
\]
we have
\begin{align*}
&  \int_{\mathbb{R}^{N}}\left\vert \nabla v_{k}-\nabla u_{0}\right\vert
^{N}dx+\int_{\mathbb{R}^{N}}V(x)\left\vert v_{k}-u_{0}\right\vert ^{N}\\
&  \leq C_{1}\int_{\mathbb{R}^{N}}\left(  \left\vert \nabla v_{k}\right\vert
^{N-2}\nabla v_{k}-\left\vert \nabla u_{0}\right\vert ^{N-2}\nabla
u_{0}\right)  \left(  \nabla v_{k}-\nabla u_{0}\right)  dx\\
&  +C_{2}\int_{\mathbb{R}^{N}}V(x)\left(  \left\vert v_{k}\right\vert
^{N-2}v_{k}-\left\vert u_{0}\right\vert ^{N-2}u_{0}\right)  \left(
v_{k}-u_{0}\right)  dx
\end{align*}
where we did use the inequality $\left(  \left\vert x\right\vert
^{N-2}x-\left\vert y\right\vert ^{N-2}y\right)  \left(  x-y\right)
\geq2^{2-N}\left\vert x-y\right\vert ^{N}.$ So we can conclude that
$v_{k}\rightarrow u_{0}$ in $E$. Thus $J_{\varepsilon}\left(  v_{k}\right)
\rightarrow J_{\varepsilon}\left(  u_{0}\right)  =c_{0}<0$. Again, this is a
contradiction. The proof is thus complete.
\end{proof}

\section{The existence result to the problem (\ref{1.3})}

In this section, we deal with the problem (\ref{1.3}). The main result of ours
shows that we don't need a nonzero small perturbation in this case to
guarantee the existence of a solution.

\subsection{Proof of Theorem 2.3}

It's similar to the proof of Theorems 2.1 and 2.2. We can find a sequence
$\left(  v_{k}\right)  $ in $E$ such that
\[
J\left(  v_{k}\right)  \rightarrow c_{M}>0\text{ and }\left(  1+\left\Vert
v_{k}\right\Vert _{E}\right)  \left\Vert DJ\left(  v_{k}\right)  \right\Vert
\rightarrow0
\]
where $c_{M}$ is the mountain-pass level of $J$. Now, by Lemma 4.1, the
sequence $\left(  v_{k}\right)  $ converges weakly to a weak solution $v$ of
(\ref{1.3}) in $E$. Now, suppose that $v=0$. Similarly as in the proof of
Proposition 5.2, we have that:
\begin{equation}
\int_{\mathbb{R}^{N}}\frac{F\left(  x,v_{k}\right)  }{\left\vert x\right\vert
^{\beta}}\rightarrow0\label{m1}%
\end{equation}
So
\[
\underset{k\rightarrow\infty}{\lim}\left\Vert v_{k}\right\Vert _{E}%
^{N}=\underset{k\rightarrow\infty}{\lim}\left(  NJ\left(  v_{k}\right)
+N\int\limits_{\mathbb{R}^{N}}\frac{F(x,v_{k})}{\left\vert x\right\vert
^{\beta}}dx\right)  =NC_{M}%
\]
Note that by Lemma 3.6, we have $0<C_{M}<\frac{1}{N}\left(  \frac{N-\beta}%
{N}\frac{\alpha_{N}}{\alpha_{0}}\right)  ^{N-1}$, so
\[
\underset{k\rightarrow\infty}{\lim\sup}\left\Vert v_{k}\right\Vert
_{E}<\left(  \frac{N-\beta}{N}\frac{\alpha_{N}}{\alpha_{0}}\right)  ^{\left(
N-1\right)  /N}.
\]
Thus by $(f1)$, we have
\[
\frac{f(x,v_{k})v_{k}}{\left\vert x\right\vert ^{\beta}}\leq b_{1}\frac
{v_{k}^{N}}{\left\vert x\right\vert ^{\beta}}+b_{2}\frac{R\left(  \alpha
_{0},v_{k}\right)  v_{k}}{\left\vert x\right\vert ^{\beta}}%
\]
Note that%
\[
b_{1}\int\limits_{\mathbb{R}^{N}}\frac{v_{k}^{N}}{\left\vert x\right\vert
^{\beta}}+b_{2}\int\limits_{\mathbb{R}^{N}}\frac{R\left(  \alpha_{0}%
,v_{k}\right)  v_{k}}{\left\vert x\right\vert ^{\beta}}\rightarrow0
\]
since by Lemma 3.2 and by the compact embedding $E\hookrightarrow L^{s}\left(
\mathbb{R}^{N}\right)  ,~s\geq N$, $\int\limits_{\mathbb{R}^{N}}\frac{R\left(
\alpha_{0},v_{k}\right)  v_{k}}{\left\vert x\right\vert ^{\beta}}\leq C\left(
M,N\right)  \left\Vert v_{k}\right\Vert _{s}\rightarrow0.$ Moreover,
$\int\limits_{\left\vert x\right\vert \geq1}\frac{v_{k}^{N}}{\left\vert
x\right\vert ^{\beta}}\leq\left\Vert v_{k}\right\Vert _{N}^{N}\rightarrow0$
again by the compact embedding $E\hookrightarrow L^{N}\left(  \mathbb{R}%
^{N}\right)  $ and $\int\limits_{\left\vert x\right\vert \leq1}\frac{v_{k}%
^{N}}{\left\vert x\right\vert ^{\beta}}\leq C\left\Vert v_{k}\right\Vert
_{Nr}^{N}\rightarrow0$ by Holder's inequality and by the compact embedding
$E\hookrightarrow L^{s}\left(  \mathbb{R}^{N}\right)  ,~s\geq N$. So we can
conclude that
\[
\int\limits_{\mathbb{R}^{N}}\frac{f(x,v_{k})v_{k}}{\left\vert x\right\vert
^{\beta}}dx\rightarrow0
\]
which thus $\underset{k\rightarrow\infty}{\lim}\left\Vert v_{k}\right\Vert
_{E}^{N}=\underset{k\rightarrow\infty}{\lim}\int\limits_{\mathbb{R}^{N}}%
\frac{f(x,v_{k})v_{k}}{\left\vert x\right\vert ^{\beta}}dx=0$ and it's
impossible. So we get the nontriviality of the solution.

\section{Existence and Multiplicity Without the Ambrosetti-Rabinowitz
condition}

The main purpose of this section is to prove that all of the results of
existence and multiplicity in Sections 5 and 6 hold even when the nonlinear
term $f$ does not satisfy the Ambrosetti-Rabinowitz condition. It is not
difficult to see that there are many interesting examples of such $f$ which do
not satisfy the Ambrosetti-Rabinowitz condition, but satisfy our weaker
conditions listed below.

In this section, instead of conditions $(f2)$ and $(f3)$, we assume that

$(f2^{\prime})$ $H(x,t)\leq H(x,s)$ for all $0<t<s$, $\forall x\in
\mathbb{R}^{N}$ where $H(x,u)=uf(x,u)-NF(x,u)$.

$(f3^{\prime})$ There exists $c>0$ such that for all $\left(  x,s\right)
\in\mathbb{R}^{N}\times\mathbb{R}^{+}:F(x,s)\leq c\left\vert s\right\vert
^{N}+cf(x,s)$.

$(f4^{\prime})$ $\underset{u\rightarrow\infty}{\lim}\frac{F(x,u)}{\left\vert
u\right\vert ^{N}}=\infty$ uniformly on \ $x\in$ $\mathbb{R}^{N}$.

We should stress that $\left(  f1\right)  +(f3)$ will imply $(f3^{\prime})$.

The key to establish the results in earlier sections is to prove that the
Cerami sequence \cite{Ce1, Ce2} associated to the Lagrange-Euler functional is
bounded. Once we will have proved this, the remaining should be the same as in
previous sections. Therefore, we only include the proof of this essential
ingredient in this section.

\begin{lemma}
Let $\left\{  u_{k}\right\}  $ be an arbitrary Cerami sequence associated to
the functional
\[
I(u)=\frac{1}{N}\left\Vert u\right\Vert ^{N}-\int\limits_{\mathbb{R}^{N}}%
\frac{F(x,u)}{\left\vert x\right\vert ^{\beta}}dx
\]
such that
\begin{align*}
\frac{1}{N}\left\Vert u_{k}\right\Vert ^{N}-\int\limits_{\mathbb{R}^{N}}%
\frac{F(x,u_{k})}{\left\vert x\right\vert ^{\beta}}dx  &  \rightarrow C_{M}\\
\left(  1+\left\Vert u_{k}\right\Vert \right)  \left\vert \int
\limits_{\mathbb{R}^{N}}\left\vert \nabla u_{k}\right\vert ^{N-1}\nabla
u_{k}\nabla vdx+\int\limits_{\mathbb{R}^{N}}V(x)\left\vert u_{k}\right\vert
^{N-1}u_{k}vdx-\int\limits_{\mathbb{R}^{N}}\frac{f(x,u_{k})v}{\left\vert
x\right\vert ^{\beta}}dx\right\vert  &  \leq\varepsilon_{k}\left\Vert
v\right\Vert \\
\varepsilon_{k}  &  \rightarrow0\text{.}%
\end{align*}
where $C_{M}\in\left(  0,\frac{1}{N}\left(  \left(  1-\frac{\beta}{N}\right)
\frac{\alpha_{N}}{\alpha_{0}}\right)  ^{N-1}\right)  $. Then $\left\{
u_{k}\right\}  $ is bounded up to a subsequence.
\end{lemma}

\begin{proof}
Suppose that
\begin{equation}
\left\Vert u_{k}\right\Vert \rightarrow\infty\label{8.6}%
\end{equation}
Setting
\[
v_{k}=\frac{u_{k}}{\left\Vert u_{k}\right\Vert }%
\]
then $\left\Vert v_{k}\right\Vert =1$. We can then suppose that $v_{k}%
\rightharpoonup v$ in $E$ (up to a subsequence) . We may similarly show that
$v_{k}^{+}\rightharpoonup v^{+}$ in $E$, where $w^{+}=\max\left\{
w,0\right\}  .$ Thanks to the assumptions on the potential $V$, the embedding
$E\hookrightarrow L^{q}\left(  \mathbb{R}^{N}\right)  $ is compact for all
$q\geq N$. So, we can assume that $\left\{
\begin{array}
[c]{l}%
v_{k}^{+}(x)\rightarrow v^{+}(x)\text{ a.e. in }\mathbb{R}^{N}\\
v_{k}^{+}\rightarrow v^{+}\text{ in }L^{q}\left(  \mathbb{R}^{N}\right)
,~\forall q\geq N
\end{array}
\right.  .$ We wish to show that $v^{+}=0$ a.e. $\mathbb{R}^{N}.$ Indeed, if
$S^{+}=\left\{  x\in%
\mathbb{R}
^{N}:v^{+}\left(  x\right)  >0\right\}  $ has a positive measure$,$ then in
$S^{+}$, we have
\[
\underset{k\rightarrow\infty}{\lim}u_{k}^{+}(x)=\underset{k\rightarrow\infty
}{\lim}v_{k}^{+}(x)\left\Vert u_{k}\right\Vert =+\infty
\]
and thus by $(f4^{\prime}):$
\[
\underset{k\rightarrow\infty}{\lim}\frac{F\left(  x,u_{k}^{+}(x)\right)
}{\left\vert x\right\vert ^{\beta}\left\vert u_{k}^{+}(x)\right\vert ^{N}%
}=+\infty\text{ a.e. in }S^{+}%
\]
This means that
\begin{equation}
\underset{n\rightarrow\infty}{\lim}\frac{F\left(  x,u_{k}^{+}(x)\right)
}{\left\vert x\right\vert ^{\beta}\left\vert u_{k}^{+}(x)\right\vert ^{N}%
}\left\vert v_{k}^{+}(x)\right\vert ^{N}=+\infty\text{ a.e. in }S^{+}
\label{8.7}%
\end{equation}
and so
\begin{equation}
\int_{\mathbb{R}^{N}}\underset{k\rightarrow\infty}{\lim\inf}\frac{F\left(
x,u_{k}^{+}(x)\right)  }{\left\vert x\right\vert ^{\beta}\left\vert u_{k}%
^{+}(x)\right\vert ^{N}}\left\vert v_{k}^{+}(x)\right\vert ^{N}dx=+\infty
\label{8.8}%
\end{equation}
However, since $\left\{  u_{k}\right\}  $ is the arbitrary Cerami sequence at
level $C_{M}$, we see that
\[
\left\Vert u_{k}\right\Vert ^{N}=NC_{M}+N\int_{%
\mathbb{R}
^{N}}\frac{F(x,u_{k}^{+}\left(  x\right)  )}{\left\vert x\right\vert ^{\beta}%
}dx+o(1)
\]
which implies that
\[
\int_{%
\mathbb{R}
^{N}}\frac{F(x,u_{k}^{+}\left(  x\right)  )}{\left\vert x\right\vert ^{\beta}%
}dx\rightarrow+\infty
\]
and then
\begin{align}
&  \underset{k\rightarrow\infty}{\lim\inf}\int_{%
\mathbb{R}
^{N}}\frac{F\left(  x,u_{k}^{+}(x)\right)  }{\left\vert x\right\vert ^{\beta
}\left\vert u_{k}^{+}(x)\right\vert ^{N}}\left\vert v_{k}^{+}(x)\right\vert
^{N}dx\label{8.9}\\
&  \underset{k\rightarrow\infty}{=\lim\inf}\int_{%
\mathbb{R}
^{N}}\frac{F\left(  x,u_{k}^{+}(x)\right)  }{\left\vert x\right\vert ^{\beta
}\left\Vert u_{k}\right\Vert ^{N}}dx\nonumber\\
&  =\underset{k\rightarrow\infty}{\lim\inf}\frac{\int_{%
\mathbb{R}
^{N}}\frac{F\left(  x,u_{k}^{+}(x)\right)  }{\left\vert x\right\vert ^{\beta}%
}dx}{NC_{M}+N\int_{%
\mathbb{R}
^{N}}\frac{F(x,u_{k}^{+}\left(  x\right)  )}{\left\vert x\right\vert ^{\beta}%
}dx+o(1)}\nonumber\\
&  =\frac{1}{N}\nonumber
\end{align}
Now, note that $F(x,s)\geq0$, by Fatou's lemma and (\ref{8.8}) and
(\ref{8.9}), we get a contradiction. So $v\leq0$ a.e. which means that
$v_{k}^{+}\rightharpoonup0$ in $E.$

Letting $t_{k}\in\left[  0,1\right]  $ such that
\[
I\left(  t_{k}u_{k}\right)  =\underset{t\in\left[  0,1\right]  }{\max}I\left(
tu_{k}\right)
\]
For any given $R\in\left(  0,\left(  \frac{\left(  1-\frac{\beta}{N}\right)
\alpha_{N}}{\alpha_{0}}\right)  ^{\frac{N-1}{N}}\right)  $, let $\varepsilon
=\frac{\left(  1-\frac{\beta}{N}\right)  \alpha_{N}}{R^{N/(N-1)}}-$
$\alpha_{0}>0$, since $f$ has critical growth $(f1)$ on $%
\mathbb{R}
^{N}$, there exists $C=C(R)>0$ such that
\begin{equation}
F(x,s)\leq C\left\vert s\right\vert ^{N}+\left\vert \frac{\left(
1-\frac{\beta}{N}\right)  \alpha_{N}}{R^{N/(N-1)}}-\alpha_{0}\right\vert
R\left(  \alpha_{0}+\varepsilon,s\right)  ,~\forall\left(  x,s\right)  \in%
\mathbb{R}
^{N}\times%
\mathbb{R}
. \label{8.10}%
\end{equation}
Since $\left\Vert u_{k}\right\Vert \rightarrow\infty$, we have
\begin{equation}
I\left(  t_{k}u_{k}\right)  \geq I\left(  \frac{R}{\left\Vert u_{k}\right\Vert
}u_{k}\right)  =I\left(  Rv_{k}\right)  \label{8.11}%
\end{equation}
and by (\ref{8.10}), $\left\Vert v_{k}\right\Vert =1$ and the fact that
$\int_{%
\mathbb{R}
^{N}}\frac{F\left(  x,v_{k}\right)  }{\left\vert x\right\vert ^{\beta}}%
dx=\int_{%
\mathbb{R}
^{N}}\frac{F\left(  x,v_{k}^{+}\right)  }{\left\vert x\right\vert ^{\beta}}%
dx$, we get
\begin{align}
&  NI\left(  Rv_{k}\right) \label{8.12}\\
&  \geq R^{N}-NCR^{N}\int_{%
\mathbb{R}
^{N}}\frac{\left\vert v_{k}^{+}\right\vert ^{N}}{\left\vert x\right\vert
^{\beta}}dx-N\left\vert \frac{\left(  1-\frac{\beta}{N}\right)  \alpha_{N}%
}{R^{\frac{N}{N-1}}}-\alpha_{0}\right\vert \int_{%
\mathbb{R}
^{N}}\frac{R\left(  \alpha_{0}+\varepsilon,R\left\vert v_{k}^{+}\right\vert
\right)  }{\left\vert x\right\vert ^{\beta}}dx\nonumber\\
&  \geq R^{N}-NCR^{N}\int_{%
\mathbb{R}
^{N}}\frac{\left\vert v_{k}^{+}\right\vert ^{N}}{\left\vert x\right\vert
^{\beta}}dx-N\left\vert \frac{\left(  1-\frac{\beta}{N}\right)  \alpha_{N}%
}{R^{\frac{N}{N-1}}}-\alpha_{0}\right\vert \int_{%
\mathbb{R}
^{N}}\frac{R\left(  \left(  \alpha_{0}+\varepsilon\right)  R^{\frac{N}{N-1}%
},\left\vert v_{k}^{+}\right\vert \right)  }{\left\vert x\right\vert ^{\beta}%
}dx\nonumber\\
&  \geq R^{N}-NCR^{N}\int_{%
\mathbb{R}
^{N}}\frac{\left\vert v_{k}^{+}\right\vert ^{N}}{\left\vert x\right\vert
^{\beta}}dx-N\left\vert \frac{\left(  1-\frac{\beta}{N}\right)  \alpha_{N}%
}{R^{\frac{N}{N-1}}}-\alpha_{0}\right\vert \int_{%
\mathbb{R}
^{N}}\frac{R\left(  \left(  1-\frac{\beta}{N}\right)  \alpha_{N},\left\vert
v_{k}\right\vert \right)  }{\left\vert x\right\vert ^{\beta}}dx\nonumber
\end{align}
Since $v_{k}^{+}\rightharpoonup0$ in $E$ and the embedding $E\hookrightarrow
L^{p}\left(
\mathbb{R}
^{N}\right)  $ is compact for all $p\geq N$, using the Holder inequality, we
can show easily that $\int_{\mathbb{R}^{N}}\frac{\left\vert v_{k}%
^{+}(x)\right\vert ^{N}}{\left\vert x\right\vert ^{\beta}}dx\overset
{k\rightarrow\infty}{\rightarrow}0$. Also, by Lemma 1.1, $\int_{%
\mathbb{R}
^{N}}\frac{R\left(  \left(  1-\frac{\beta}{N}\right)  \alpha_{N},\left\vert
v_{k}(x)\right\vert \right)  }{\left\vert x\right\vert ^{\beta}}dx$ is bounded
by a universal $C$.

Thus using (\ref{8.11}) and letting $k\rightarrow\infty$ in (\ref{8.12}), and
then letting $R\rightarrow\left[  \left(  \frac{\left(  1-\frac{\beta}%
{N}\right)  \alpha_{N}}{\alpha_{0}}\right)  ^{\frac{N-1}{N}}\right]  ^{-}$, we
get
\begin{equation}
\underset{k\rightarrow\infty}{\lim\inf}I\left(  t_{k}u_{k}\right)  \geq
\frac{1}{N}\left(  \left(  1-\frac{\beta}{N}\right)  \frac{\alpha_{N}}%
{\alpha_{0}}\right)  ^{N-1}>C_{M} \label{8.13}%
\end{equation}
Note that $I(0)=0$ and $I(u_{k})\rightarrow C_{M}$, we can suppose that
$t_{k}\in\left(  0,1\right)  $. Thus since $DI(t_{k}u_{k})t_{k}u_{k}=0,$%
\[
t_{k}^{N}\left\Vert u_{k}\right\Vert ^{N}=\int_{%
\mathbb{R}
^{N}}\frac{f\left(  x,t_{k}u_{k}\right)  t_{k}u_{k}}{\left\vert x\right\vert
^{\beta}}dx
\]
By $(f2^{\prime}):$
\begin{align*}
NI\left(  t_{k}u_{k}\right)   &  =t_{k}^{N}\left\Vert u_{k}\right\Vert
^{N}-N\int_{%
\mathbb{R}
^{N}}\frac{F\left(  x,t_{k}u_{k}\right)  }{\left\vert x\right\vert ^{\beta}%
}dx\\
&  =\int_{%
\mathbb{R}
^{N}}\frac{\left[  f\left(  x,t_{k}u_{k}\right)  t_{k}u_{k}-NF\left(
x,t_{k}u_{k}\right)  \right]  }{\left\vert x\right\vert ^{\beta}}dx\\
&  \leq\int_{%
\mathbb{R}
^{N}}\frac{\left[  f\left(  x,u_{k}\right)  u_{k}-NF\left(  x,u_{k}\right)
\right]  }{\left\vert x\right\vert ^{\beta}}dx.
\end{align*}
Moreover, we have%
\begin{align*}
\int_{%
\mathbb{R}
^{N}}\frac{\left[  f\left(  x,u_{k}\right)  u_{k}-NF\left(  x,u_{k}\right)
\right]  }{\left\vert x\right\vert ^{\beta}}dx  &  =\left\Vert u_{k}%
\right\Vert ^{N}+NC_{M}-\left\Vert u_{k}\right\Vert ^{N}+o(1)\\
&  =NC_{M}+o(1)
\end{align*}
which is a contraction to (\ref{8.13}). This proves that $\left\{
u_{k}\right\}  $ is bounded in $E$.
\end{proof}

\end{document}